\documentclass[12pt,reqno]{amsart}

\usepackage{amstext}
\usepackage{amsthm}
\usepackage{amsmath}
\usepackage{amssymb}
\usepackage{latexsym}
\usepackage{amsfonts}
\usepackage{graphicx}
\usepackage{texdraw}
\usepackage{graphpap}

\usepackage[pagebackref,hypertexnames=false, colorlinks, citecolor=red, linkcolor=red]{hyperref} %

\usepackage[backrefs]{amsrefs}

\input txdtools

\begin{document}

\newcommand{\ci}[1]{_{ {}_{\scriptstyle #1}}}

\newcommand{\norm}[1]{\ensuremath{\left\|#1\right\|}}
\newcommand{\abs}[1]{\ensuremath{\left\vert#1\right\vert}}
\newcommand{\ip}[2]{\ensuremath{\left\langle#1,#2\right\rangle}}
\newcommand{\p}{\ensuremath{\partial}}
\newcommand{\pr}{\mathcal{P}}

\newcommand{\pbar}{\ensuremath{\bar{\partial}}}
\newcommand{\db}{\overline\partial}
\newcommand{\D}{\mathbb{D}}
\newcommand{\B}{\mathbb{B}}
\newcommand{\Sp}{\mathbb{S}}
\newcommand{\T}{\mathbb{T}}
\newcommand{\R}{\mathbb{R}}
\newcommand{\Z}{\mathbb{Z}}
\newcommand{\C}{\mathbb{C}}
\newcommand{\N}{\mathbb{N}}
\newcommand{\scrH}{\mathcal{H}}
\newcommand{\scrL}{\mathcal{L}}
\newcommand{\td}{\widetilde\Delta}

\newcommand{\La}{\langle }
\newcommand{\Ra}{\rangle }
\newcommand{\rk}{\operatorname{rk}}
\newcommand{\card}{\operatorname{card}}
\newcommand{\ran}{\operatorname{Ran}}
\newcommand{\osc}{\operatorname{OSC}}
\newcommand{\im}{\operatorname{Im}}
\newcommand{\re}{\operatorname{Re}}
\newcommand{\tr}{\operatorname{tr}}
\newcommand{\vf}{\varphi}
\newcommand{\f}[2]{\ensuremath{\frac{#1}{#2}}}

\newcommand{\kzp}{k_z^{(p,\alpha)}}
\newcommand{\klp}{k_{\lambda_i}^{(p,\alpha)}}
\newcommand{\TTp}{\mathcal{T}_p}


\newcommand{\entrylabel}[1]{\mbox{#1}\hfill}

\newenvironment{entry}
{\begin{list}{X}%
  {\renewcommand{\makelabel}{\entrylabel}%
      \setlength{\labelwidth}{55pt}%
      \setlength{\leftmargin}{\labelwidth}
      \addtolength{\leftmargin}{\labelsep}%
   }%
}%
{\end{list}}


\numberwithin{equation}{section}

\newtheorem{thm}{Theorem}[section]
\newtheorem{lm}[thm]{Lemma}
\newtheorem{cor}[thm]{Corollary}
\newtheorem{conj}[thm]{Conjecture}
\newtheorem{prob}[thm]{Problem}
\newtheorem{prop}[thm]{Proposition}
\newtheorem*{prop*}{Proposition}

\theoremstyle{remark}
\newtheorem{rem}[thm]{Remark}
\newtheorem*{rem*}{Remark}

\hyphenation{geo-me-tric}

\title{The Essential Norm of Operators on $A^p_\alpha(\mathbb{B}_n)$}

\author[M. Mitkovski]{Mishko Mitkovski$^\dagger$}
\address{Mishko Mitkovski, School of Mathematics\\ Georgia Institute of Technology\\ 686 Cherry Street\\ Atlanta, GA USA 30332-0160}
\email{mitkovski@math.gatech.edu}
\thanks{$\dagger$ Research supported in part by National Science Foundation DMS grant \# 1101251.}

\author[D. Su\'arez]{Daniel Su\'arez$^\sharp$}
\address{Daniel Su\'{a}rez, Depto. de Matem\'{a}tica\\ FCEyN, University of Buenos Aires\\ Pab. I, Ciudad Universitaria
\\ (1428) N\'{u}\~{n}ez, Capital Federal\\ Argentina}
\email{dsuarez@dm.uba.ar}
\thanks{$\sharp$ Research supported in part by the ANPCyT grant PICT2009-0082, Argentina.}

\author[B. D. Wick]{Brett D. Wick$^\ddagger$}
\address{Brett D. Wick, School of Mathematics\\ Georgia Institute of Technology\\ 686 Cherry Street\\ Atlanta, GA USA 30332-0160}
\email{wick@math.gatech.edu}
\thanks{$\ddagger$ Research supported in part by National Science Foundation DMS grants \# 1001098 and \# 955432.}


\begin{abstract}
In this paper we characterize the compact operators on $A^p_\alpha(\mathbb{B}_n)$ when $1<p<\infty$ and $\alpha>-1$.  The main result shows that an operator on $A^p_\alpha(\mathbb{B}_n)$ is compact if and only if it belongs to the Toeplitz algebra and its Berezin transform vanishes on the boundary of the ball.
\end{abstract}

\maketitle

\section{Introduction and Statement of Main Results}

Let $\B_n$ denote the unit ball in $\C^n$.  For $\alpha>-1$, we let
$$
dv_\alpha(z) := c_\alpha  \, (1-|z|^2)^\alpha \, dv(z),
\ \mbox{ with }\
c_\alpha := \frac{\Gamma(n+\alpha+1)}{n! \,\Gamma(\alpha+1)}.
$$
This choice of $c_\alpha$ gives that $v_\alpha\left(\B_n\right)=1$.  For $1<p<\infty$ the space $A^p_\alpha(\B_n):=A^p_\alpha$ is the collection of all holomorphic functions on $\B_n$ such that
$$
\norm{f}_{A^p_{\alpha}}^p:=\int_{\B_n}\abs{f(z)}^p\,dv_{\alpha}(z)<\infty.
$$
We will also let $L^p_\alpha(\B_n):=L^p_\alpha$ denote the standard Lebesgue space on $\B_n$ with respect to the measure $v_\alpha$.

Recall that the projection of $L^2_\alpha$ onto $A^2_\alpha$ is given by the integral operator
$$
P_\alpha(f)(z):=\int_{\B_n}\frac{f(w)}{(1-z\overline{w})^{n+1+\alpha}}\,dv_{\alpha}(w).
$$
It is well-known that this operator is bounded from $L^p_\alpha$ to $A^p_\alpha$ when $1<p<\infty$ and $\alpha>-1$.  Let $M_a$ denote the operator of multiplication by the function $a$, $M_a(f):=af$.  The Toeplitz operator with symbol $a\in L^\infty$ is then defined by
$$
T_a:=P_{\alpha} M_a.
$$
It is immediate to see that $\norm{T_a}_{\mathcal{L}(L^p_\alpha, A^p_\alpha)}\lesssim\norm{a}_{L^\infty}$.   For $1<p<\infty$, $\alpha>-1$ and for $\lambda\in\B_n$ let $k_{\lambda}^{(p,\alpha)}(z)=\frac{(1-\abs{\lambda}^2)^{\frac{n+1+\alpha}{q}}}{(1-\overline{\lambda}z)^{n+1+\alpha}}$, where as usual $q=\frac{p}{(p-1)}$.  We also let $K_\lambda(z)=\frac{1}{(1-\overline{\lambda}z)^{n+1+\alpha}}$, which is the standard reproducing kernel in the space $A^2_\alpha$.

The Berezin transform of an operator $S$ on $A^p_\alpha$ is defined by
$$
B(S)(z):=\ip{Sk_z^{(p,\alpha)}}{k_z^{(q,\alpha)}}_{A^2_\alpha}.
$$
It is easy to see that if $S$ is a bounded operator then $\sup \{ |B(S)(z)| : z\in \B_n \} \lesssim \|S\|$.
One of the interesting aspects of operator theory on the Bergman space is that the Berezin transform essentially encapsulates
all the behavior of the operator. In fact, the Berezin transform is one-to-one, so every bounded operator on $A^p_\alpha$
is determined by its Berezin transform $B(S)$.  It is also easy to see that if $S$ is compact,
then $B(S)(z)\to 0$ as $\abs{z}\to 1$. Moreover, as we will see in this paper, it is possible to obtain a characterization
of compact operators on $A^p_\alpha$ in terms of the Berezin transform.
The following papers provide additional examples of how the Berezin transform determines properties of several classes of
operators on the Bergman space of the unit ball $\B_n$, \cites{I, NZZ, LH, YS, AZ2,SZ, St}.

As motivation for our project, we highlight some of the major contributions leading to a characterization of
compactness in terms of the Berezin transform.  A major breakthrough was obtained by Axler and Zheng
for the standard Bergman space $A^2_0(\D)$, see \cite{AZ}.
They showed that if $S$ is a finite sum of finite products of Toeplitz operators,
$S$ is compact if and only if the Berezin transform vanishes as $\abs{z}\to 1$.
This was later extended by Engli{\v{s}} to the case of bounded symmetric domains in $\C^n$, see \cite{E}.
See also the proof by Raimondo, \cite{R}, in the specific case of $\B_n$.

To state the next contribution, we need a little more notation.  Let $\mathcal{T}_{p,\alpha}$ denote the Toeplitz algebra
generated by $L^\infty$ functions.  Miraculously, there is a very close relationship between membership in
$\mathcal{T}_{p,\alpha}$ and compactness since it is known that the compact operators
on $A^p_\alpha$ belong to $\mathcal{T}_{p,\alpha}$, see \cite{E92}.  When $\alpha=0$, the second author showed in \cite{Sua}
that the compact operators are precisely those that belong to the Toeplitz algebra and have a Berezin transform that vanishes
on the boundary of the unit ball.
The main theorem of this paper is a generalization of the last result to $\alpha > -1$, as stated below.
\begin{thm}
Let $1<p<\infty$ and $\alpha>-1$ and $S\in\mathcal{L}(A^p_{\alpha},A^p_{\alpha})$.  Then $S$ is compact if and only if $S\in\mathcal{T}_{p,\alpha}$ and $\lim_{|z|\rightarrow 1} B(S)(z)=0$.
\end{thm}

There is a well-known similarity between results on the Bergman space $A^p_\alpha$ and the Fock space of entire functions
$\mathcal{F}^p_\alpha(\C^n)$.  These are the entire functions on $\C^n$ such that
$$
\int_{\C^n}\abs{f(z) e^{-\frac{\alpha}{2}\abs{z}^2}}^p\,dv(z)<\infty.
$$
In the recent paper \cite{BI}, Bauer and Isralowitz obtained analogous results for the compact operators on the Fock space.
In particular, they showed that an operator on the Fock space is compact if and only if it belongs to the Toeplitz algebra and the Berezin transform vanishes at infinity.


The outline of the paper is as follows.  In Section \ref{Prelim} we fix notation and state some additional facts that will
be needed throughout the paper. In Section~\ref{Approximation} we show how to approximate $S\in\mathcal{T}_{p,\alpha}$
by certain localized operators that will be crucial when computing the essential norm of $S$.
In Section~\ref{UniformAlg} we introduce a way to connect the behavior of the Berezin transform to the behavior
of these localized operators.
Finally, in Section \ref{Characterization} we merge the ingredients of the two previous sections to prove our main results.
This is accomplished by obtaining several different characterizations of the essential norm of an operator on $A^p_{\alpha}$.

Throughout this paper we use the standard notation $A\lesssim B$ to denote the existence of a constant $C$ such that $A\leq C B$.  While $A\approx B$ will mean $A\lesssim B$ and $B\lesssim A$.  The value of a constant may change from line to line, but we will frequently attempt to denote the parameters the constant depends upon.  The expression $:=$ will mean equal by definition.

\section{Preliminaries}
\label{Prelim}
We let $\overline{z}w$ denote the standard inner product in $\C^n$.  For $z\in\B_n$, $\varphi_z$ will denote the
involutive automorphism of $\B_n$ such that $\varphi_z(0)=z$.  Using this automorphism, the pseudohyperbolic and hyperbolic metrics on $\B_n$ are defined by
$$
\rho(z,w):=\abs{\varphi_z(w)}\quad\textnormal{ and }\quad \beta(z,w):=\frac{1}{2}\log\frac{1+\rho(z,w)}{1-\rho(z,w)}.
$$
Recall that these metrics are connected by $\rho=\frac{e^{2\beta}-1}{e^{2\beta}+1}=\tanh\beta$.  It is well-known that these metrics are invariant under the automorphism group of $\B_n$.  We let
$$
D(z,r):=\{w\in\B_n:\beta(z,w)\leq r\}=\{w\in\B_n: \rho(z,w)\leq s=\tanh r\},
$$
denote the hyperbolic disc centered at $z$ of radius $r$.  Recall the following well-known identity for the M\"obius maps that will be used many times in what follows:
$$
1-\abs{\varphi_z(w)}^2=\frac{(1-\abs{z}^2)(1-\abs{w}^2)}{\abs{1-\overline{z}w}^2}.
$$

For $1<p<\infty$, $-1<\alpha$, and for $\lambda\in\B_n$,
if $k_{\lambda}^{(p,\alpha)}(z)=\frac{(1-\abs{\lambda}^2)^{\frac{n+1+\alpha}{q}}}{(1-\overline{\lambda}z)^{n+1+\alpha}}$,
we have that $\norm{k_{\lambda}^{(p,\alpha)}}_{A^p_{\alpha}}\approx 1$ with implied constants depending on $p,\alpha, n$.
For a set $E\subset\B_n$, we let $1_E$ denote the indicator function of the set $E$.

The next lemma is well-known, and we omit the proof.  The interested reader can consult the book \cite{Zhu}.
\begin{lm}
\label{Growth}
For $z\in \B_n$, $s$ real and $t>-1$, let
$$
F_{s,t}(z):=\int_{\B_n}\frac{(1-\abs{w}^2)^t}{\abs{1-\overline{w}z}^s}\,dv(w).
$$
Then $F_{s,t}$ is bounded if $s<n+1+t$ and grows as $(1-\abs{z}^2)^{n+1+t-s}$ when $\abs{z}\to 1$ if $s>n+1+t$.
\end{lm}

\subsection{Carleson Measures for \texorpdfstring{$A^p_{\alpha}\ $}{Weighted Bergman Spaces}}

Unless stated otherwise, a measure will always be a positive, finite, regular, Borel measure.  For $p\geq 1$ a measure $\mu$ on $\B_n$ is a Carleson measure for $A^p_\alpha$ if there is a constant $C_p$, independent of $f$, such that
\begin{equation}
\label{Carl_Emb}
\left(\int_{\B_n}\abs{f(z)}^p \,d\mu(z)\right)^{\frac{1}{p}}\leq C_p\left( \int_{\B_n}\abs{f(z)}^p \,dv_\alpha(z)\right)^{\frac{1}{p}}.
\end{equation}
The best constant $C_p$ such that \eqref{Carl_Emb} holds will be denoted by $\norm{\imath_p}$.

For a measure $\mu$ we  define the operator
$$
T_\mu f(z):=\int_{\B_n}\frac{f(w)}{(1-\overline{w}z)^{n+1+\alpha}}\,d\mu(w),
$$
which gives rise to an analytic function for all $f\in H^\infty$.  When $1<p<\infty$, $T_\mu$ is densely defined on $A^p_\alpha$, and $T_\mu$ is bounded from $A^p_{\alpha}\to A^p_{\alpha}$ if and only if $\mu$ is a Carleson measure for $A^p_{\alpha}$.  Notice also that if $\mu$ is absolutely continuous measure with density $a$, i.e., if $d\mu(z)=a(z)\,dv_{\alpha}(z)$ then $T_{\mu}$ is equal to the Toeplitz operator $T_{a}$.

The following well-known result provides a geometric characterization of the Carleson measures for $A^p_{\alpha}$.

\begin{lm}[Necessary and Sufficient Conditions for $A^p_{\alpha}$ Carleson Measures]
\label{CM}
Suppose that $1<p<\infty$ and $\alpha>-1$.  Let $\mu$ be a measure on $\B_n$ and $r>0$.  The following quantities are equivalent, with constants that depend on $n$, $\alpha$ and $r$:
\begin{itemize}
\item[(1)] $\norm{\mu}_{\textnormal{RKM}}:=\sup_{z\in\B_n}\int_{\B_n} \frac{(1-\abs{z}^2)^{n+1+\alpha}}{\abs{1-\overline{z} w}^{2(n+1+\alpha)}}\,d\mu(w)$;
\item[(2)] $\norm{\imath_p}^p:=\inf\left\{C: \int_{\B_n}\abs{f(z)}^p \,d\mu(z)\leq C \int_{\B_n}\abs{f(z)}^p \,dv_{\alpha}(z)\right\}$;
\item[(3)] $\norm{\mu}_{\textnormal{Geo}}=\sup_{z\in\B_n}\frac{\mu\left(D(z,r)\right)}{\left(1-\abs{z}^2\right)^{n+1+\alpha}}$;
\item[(4)] $\norm{T_\mu}_{\mathcal{L}(A^p_{\alpha}, A^p_{\alpha})}$.
\end{itemize}
\end{lm}
Observe that condition (1) and (3) are actually independent of the exponent $p=2$ and so, the equivalence with (2) is actually true for all $1<p<\infty$.

Another simple observation one should make at this point is the following.  Suppose that $\mu$ is a complex-valued measure
such that $\abs{\mu}$, the variation of the measure, is a Carleson measure.  Decompose $\mu$ into its real and imaginary parts and then use the Jordan Decomposition to write $\mu=\mu_1-\mu_2+i\mu_3-i\mu_4$ where each $\mu_j$ is a positive measure and $\abs{\mu}\approx \sum_{j=1}^{4}\abs{\mu_j}$.  Then $\abs{\mu_j}$ is a Carleson measure with $\norm{\abs{\mu}}_{\textnormal{RKM}}\approx\sum_{j=1}^{4}\norm{\mu_j}_{\textnormal{RKM}}$.  Using Lemma \ref{CM} we have that $T_\mu$ is a bounded operator on $A^p_{\alpha}$ when $\mu$ is a complex-valued measure with $\abs{\mu}$ a Carleson measure.

\begin{proof}
The equivalence between (1), (2) and (3) is well-known, \cite{Zhu}.  Finally, to prove the equivalence with (4), first suppose that (2) holds.  Then, using Fubini's Theorem we have that for $f,g\in H^\infty$
\begin{eqnarray*}
\abs{\ip{T_\mu f}{g}_{A^2_{\alpha}}} & = & \abs{\int_{\B_n} f(w)\overline{g(w)}\,d\mu(w)}\\
& \leq & \norm{f}_{L^p_{d\mu}}\norm{g}_{L^q_{d\mu}}\\
& \lesssim & \norm{\imath_p}\norm{\imath_q}\norm{f}_{A^p_{\alpha}}\norm{g}_{A^q_{\alpha}}.
\end{eqnarray*}
But, this inequality implies $T_\mu: A^p_{\alpha}\to A^p_{\alpha}$ is bounded.  Here we have identified $\left(A^p_{\alpha}\right)^*=A^q_\alpha$.  Conversely, if $T_\mu$ is bounded, observe
$$
T_{\mu}\left(k_{\lambda}^{(p,\alpha)}\right)(z)=\int_{\B_n}\frac{1}{(1-z\overline{w})^{n+1+\alpha}}\frac{(1-\abs{\lambda}^2)^{\frac{n+1+\alpha}{q}}}{(1-\overline{\lambda}w)^{n+1+\alpha}}\,d\mu(w)
$$
and in particular
$$
T_{\mu}\left(k_{\lambda}^{(p,\alpha)}\right)(\lambda)=\int_{\B_n}\frac{(1-\abs{\lambda}^2)^{\frac{n+1+\alpha}{q}}}{\abs{1-\overline{\lambda}w}^{2(n+1+\alpha)}}\,d\mu(w).
$$
This computation implies
\begin{eqnarray*}
\int_{\B_n}\frac{(1-\abs{\lambda}^2)^{n+1+\alpha}}{\abs{1-\overline{\lambda}w}^{2(n+1+\alpha)}}\,d\mu(w)=\ip{T_{\mu} k_\lambda^{(p,\alpha)}}{k_\lambda^{(q,\alpha)}}_{A^2_{\alpha}} & \leq & \norm{T_\mu}_{\mathcal{L}(A^p_{\alpha},A^p_{\alpha})}\norm{k_\lambda^{(p,\alpha)}}_{A^p_{\alpha}}\norm{k_\lambda^{(q,\alpha)}}_{A^q_{\alpha}}\\
& \approx & \norm{T_\mu}_{\mathcal{L}(A^p_{\alpha},A^p_{\alpha})}.
\end{eqnarray*}
\end{proof}

\begin{lm}
\label{CM-Cor}
Let $1<p<\infty$ and suppose that $\mu$ is an $A^p_{\alpha}$ Carleson measure.  Let $F\subset \B_n$ be a compact set, then
$$
\norm{T_{\mu1_F}f}_{A^p_{\alpha}}\lesssim \norm{T_\mu}^{\frac{1}{q}}_{\mathcal{L}(A^p_\alpha,A^p_\alpha)}\norm{1_F f}_{L^p(\mu)},
$$
where $q=\frac{p}{p-1}$.
\end{lm}
\begin{proof}
It is clear $T_{\mu1_F}f$ is a bounded analytic function for any $f\in A^p_{\alpha}$ since $F$ is compact and $\mu$ is a finite measure.  As in the proof of the previous lemma, we have
\begin{eqnarray*}
\abs{\ip{T_{\mu 1_F} f}{g}_{A^2_{\alpha}}} & = & \abs{\int_{\B_n} 1_F(w)f(w)\overline{g(w)}\,d\mu(w)}\\
& \leq & \norm{1_F f}_{L^p(\mu)}\norm{g}_{L^q(\mu)}\\
&\lesssim &  \norm{T_\mu}^{\frac{1}{q}}_{\mathcal{L}(A^p_\alpha,A^p_\alpha)}\norm{1_F f}_{L^p(\mu)}\norm{g}_{A^q_\alpha}.
\end{eqnarray*}
Taking the supremum over $g\in A^q_{\alpha}$ we have the desired result.
\end{proof}

For a Carleson measure $\mu$ and $1<p<\infty$ and for $f\in L^p(\B_n;\mu)$ we also define
$$
P_\mu f(z):=\int_{\B_n}\frac{f(w)}{(1-\overline{w}z)^{n+1+\alpha}}\,d\mu(w).
$$
It is easy to see based on the computations above that $P_\mu$ is a bounded operator from $L^p(\B_n;\mu)$ to $A^p_{\alpha}$ and $T_\mu=P_{\mu}\circ \imath_p$.

\subsection{Geometric Decompositions of \texorpdfstring{$\B_n$}{the Unit Ball}.}

We will use the following geometric facts. The first lemma is classical and we omit its proof. The proof of the other two can be found in~\cite{Sua}.

\begin{lm}
\label{StandardGeo}
Given $\varrho>0$, there is a family of Borel sets $D_m\subset\B_n$ and points $\{w_m\}_{m=1}^{\infty}$ such that
\begin{itemize}
\item[(i)] $D\left(w_m,\frac{\varrho}{4}\right)\subset D_m\subset D\left(w_m,\varrho\right)$ for all $m$;
\item[(ii)] $D_k\cap D_l=\emptyset$ if $k\neq l$;
\item[(iii)] $\bigcup_{m} D_m=\B_n$.
\end{itemize}
\end{lm}
It is easy to see that when the radius $\varrho$ is fixed, for $w\in D_m$, then $(1-\abs{w}^2)\approx(1-\abs{w_m}^2)$ and $\abs{1-\overline{z}w}\approx \abs{1-\overline{z}w_m}$ uniformly in $z\in\B_n$.

\begin{lm}[Lemma 3.1, \cite{Sua}]
\label{SuaGeo}
There is a positive integer $N=N(n)$ such that for any $\sigma>0$ there is a covering of $\B_n$ by Borel sets $\{B_j\}$ satisfying:
\begin{itemize}
\item[(i)] $B_j\cap B_k=\emptyset$ if $j\neq k$;
\item[(ii)] every point of $\B_n$ belongs to at most $N$ sets $\Omega_\sigma(B_j)=\{z:\beta(z,B_j)\leq\sigma\}$;
\item[(iii)] there is a constant $C(\sigma)>0$ such that $\textnormal{diam}_{\beta} \,B_j\leq C(\sigma)$ for all $j$.
\end{itemize}
\end{lm}

Let $\sigma>0$ and $k$ be a non-negative integer.  Let $\{B_j\}$ be the covering of the ball satisfying the conditions of Lemma \ref{SuaGeo} with $(k+1)\sigma$ instead of $\sigma$.  For $0\leq i\leq k$ and $j\geq 1$ write
$$
F_{0,j}:=B_j\quad\textnormal{ and }\quad F_{i+1,j}:=\left\{z:\beta(z,F_{i,j})\leq\sigma\right\}.
$$
Then we have,
\begin{lm}[Corollary 3.3, \cite{Sua}]
\label{SuaGeo2}
Let $\sigma>0$ and $k$ be a non-negative integer.  For each $0\leq i\leq k$ the family of sets $\mathcal{F}_{i}=\{F_{i,j}: j\geq 1\}$ forms a covering of $\B_n$ such that
\begin{itemize}
\item[(i)] $F_{0,j_1}\cap F_{0,j_2}=\emptyset$ if $j_1\neq j_2$;
\item[(ii)] $F_{0,j}\subset F_{1,j}\subset\cdots\subset F_{k+1,j}$ for all $j$;
\item[(iii)] $\beta(F_{i,j}, F_{i+1,j}^c)\geq\sigma$ for all $0\leq i\leq k$ and $j\geq 1$;
\item[(iv)] every point of $\B_n$ belongs to no more than $N$ elements of $\mathcal{F}_{i}$;
\item[(v)] $\textnormal{diam}_{\beta} \,F_{i,j}\leq C(k,\sigma)$ for all $i,j$.
\end{itemize}
\end{lm}

\section{Approximation by Segmented Operators}
\label{Approximation}

The goal of this section is to show that every operator in the Toeplitz algebra can be approximated by certain
localized operators that are sums of compact operators. This approximation will help us to estimate the essential norm.

\begin{thm}
\label{Approx3}
Let $S\in\mathcal{T}_{p,\alpha}$, $\mu$ be a $A^p_\alpha$ Carleson measure and $\epsilon>0$.  Then there are Borel sets $F_j\subset G_j\subset\B_n$ such that
\begin{itemize}
\item[(i)] $\B_n=\cup F_j$;
\item[(ii)] $F_j\cap F_k=\emptyset$ if $j\neq k$;
\item[(iii)] each point of $\B_n$ lies in no more than $N(n)$ of the sets $G_j$;
\item[(iv)] $\textnormal{diam}_{\beta}\, G_j\leq d(p,S,\epsilon)$
\end{itemize}
and
$$
\norm{ST_\mu-\sum_{j=1}^{\infty} M_{1_{F_j}}ST_{\mu1_{G_j}}}_{\mathcal{L}(A^p_{\alpha}, L^p_{\alpha})}<\epsilon.
$$
\end{thm}

In order to prove this result we will need several technical estimates that we group into a few lemmas. The following classical test for boundedness will be used repeatedly in what follows.

\begin{lm}[Schur's Lemma]
\label{Schur}
Let $(X,\mu)$ and $(X,\nu)$ be measure spaces, $K(x,y)$ a non-negative measurable function on $X\times X$,
$1<p<\infty$ and $\frac{1}{p}+\frac{1}{q}=1$.  If $h$ is a positive function on $X$ that is measurable with respect to
$\mu$ and $\nu$, and $C_p$ and $C_q$ are positive constants such that
\begin{itemize}
\item[] $$\int_{X} K(x,y) h(y)^q\,d\nu(y)\leq C_q h(x)^q\textnormal{ for } \mu\textnormal{-almost every } x,$$
\item[] $$\int_{X} K(x,y) h(x)^p\,d\mu(x)\leq C_p h(y)^p\textnormal{ for } \nu\textnormal{-almost every } y,$$
\end{itemize}
then $Tf(x)=\int_{X} K(x,y) f(y)\,d\nu(y)$ defines a bounded operator $T:L^p(X;\nu)\to L^p(X;\mu)$ with $\norm{T}_{L^p(\nu)\to L^p(\mu)}\leq C_q^{\frac{1}{q}}C_p^{\frac{1}{p}}$.
\end{lm}

\begin{lm}
\label{Tech1}
Let $1<p<\infty$, $\alpha>-1$ and $\mu$ be a $A^p_\alpha$ Carleson measure.
Suppose that $F_j, K_j\subset\B_n$ are Borel sets such that $\{F_j\}$ are pairwise disjoint
and $\beta(F_j,K_j)>\sigma\geq 1$ for all $j$.  If\/ $0<\gamma<\min\left\{\frac{1}{p(n+1+\alpha)},\frac{p-1}{p}\right\}$, then
\begin{equation}
\label{Tech-Est1}
\int_{\B_n}\sum_{j=1}^\infty 1_{F_j}(z) 1_{K_j}(w)\frac{(1-\abs{w}^2)^{-\frac{1}{p}}}{\abs{1-\overline{z}w}^{n+1+\alpha}}\,d\mu(w)\lesssim \norm{T_\mu}_{\mathcal{L}(A^p_\alpha, A^p_\alpha)} (1-\delta^{2n})^{\gamma}(1-\abs{z}^2)^{-\frac{1}{p}}
\end{equation}
where $\delta=\tanh\frac{\sigma}{2}$ and the implied constants depend on $n, \alpha$ and $p$.
\end{lm}

\begin{proof}
Consider a sequence of points $\{w_m\}$ and Borel sets $D_m$ as in Lemma \ref{StandardGeo} with $\varrho=\frac{1}{10}$.
Standard computations show that there is a constant $C(n,p,\alpha)$ such that
\begin{equation}
\label{Geo-Obs}
\frac{(1-\abs{w}^2)^{-\frac{1}{p}}}{\abs{1-\overline{z}w}^{n+1+\alpha}}\approx \frac{(1-\abs{w_m}^2)^{-\frac{1}{p}}}{\abs{1-\overline{z}w_m}^{n+1+\alpha}}
\end{equation}
for all $w\in D_m$ and $z\in \B_n$.  By the Carleson measure condition, we have a constant $C$ such that
$$
\mu(D_m)\lesssim \norm{T_\mu}_{\mathcal{L}(A^p_\alpha, A^p_\alpha)}v_{\alpha}(D_m).
$$
If $z\in F_j$ and $w\in K_j$, with $\beta(z,w)>\sigma$, then $K_j\subset \B_n\setminus D(z,\sigma)$ and
$$
\sum_{j=1}^\infty  1_{F_j}(z) 1_{K_j}(w)\leq \sum_{j} 1_{F_j}(z) 1_{\B_n\setminus D(z,\sigma)}(w).
$$
Thus, the integral in \eqref{Tech-Est1} is controlled by
\begin{eqnarray*}
J & = & \sum_{j=1}^{\infty}1_{F_j}(z) \int_{\B_n} 1_{\B_n\setminus D(z,\sigma)} \frac{(1-\abs{w}^2)^{-\frac{1}{p}}}{\abs{1-\overline{z}w}^{n+1+\alpha}}\,d\mu(w)\\
& = & \sum_{j=1}^{\infty} 1_{F_j}(z)\int_{\B_n} 1_{D(z,\sigma)^{c}}(w) \phi(z,w)\,d\mu(w)\\
& = & \sum_{j=1}^{\infty} 1_{F_j}(z) J_z.
\end{eqnarray*}
For simplicity, in the above display we write $\phi(z,w)$ as the kernel appearing above, and $D(z,\sigma)^{c}=\B_n\setminus D(z,\sigma)$.  We now estimate each integral $J_z$.  By definition,
\begin{eqnarray*}
J_z:=\int_{\B_n} 1_{D(z,\sigma)^c}(w)\phi(z,w) \,d\mu(w) & = & \sum_{j=1}^{\infty} \int_{D_j} 1_{D(z,\sigma)^c}(w)\phi(z,w) \,d\mu(w)\\
 & \leq & \sum_{D_j\cap D(z,\sigma)^{c}\neq\emptyset} \int_{D_j}\phi(z,w)\,d\mu(w)\\
 & \approx & \sum_{D_j\cap D(z,\sigma)^{c}\neq\emptyset} \int_{D_j}\phi(z,w_j)\,d\mu(w)\\
 & \lesssim & \norm{T_\mu}_{\mathcal{L}(A^p_\alpha, A^p_\alpha)} \sum_{D_j\cap D(z,\sigma)^{c}\neq\emptyset} \int_{D_j}\phi(z,w_j)\,dv_{\alpha}(w)\\
 & \approx & \norm{T_\mu}_{\mathcal{L}(A^p_\alpha, A^p_\alpha)} \sum_{D_j\cap D(z,\sigma)^{c}\neq\emptyset} \int_{D_j}\phi(z,w)\,dv_{\alpha}(w).
\end{eqnarray*}
If $D_j\cap D(z,\sigma)^{c}\neq\emptyset$ and $w\in D_j$, then
$\beta(w,D(z,\sigma)^c)\leq\textnormal{diam}_{\beta} \,D_m\leq 2\varrho=\frac{1}{5}$, and since
$$
\beta\left(D\left(z,\frac{\sigma}{2}\right), D(z,\sigma)^{c}\right)=\frac{\sigma}{2}\geq\frac{1}{2},
$$
we have that $D_j\cap D\left(z,\frac{\sigma}{2}\right)=\emptyset$ whenever $D_j\cap D(z,\sigma)^{c}\neq\emptyset$.  Thus,
$$
J_z \lesssim \norm{T_\mu}_{\mathcal{L}(A^p_\alpha, A^p_\alpha)} \sum_{D_j\cap D(z,\sigma)^{c}\neq\emptyset}
\int_{D_j}\phi(z,w)\,dv_{\alpha}(w)=\norm{T_\mu}_{\mathcal{L}(A^p_\alpha, A^p_\alpha)}
\int_{\B_n}1_{D\left(z,\frac{\sigma}{2}\right)^c}(w)\phi(z,w)\,dv_{\alpha}(w).
$$
Continuing the estimate, we have
\begin{eqnarray*}
J   & \lesssim &  \norm{T_\mu}_{\mathcal{L}(A^p_\alpha, A^p_\alpha)}  \sum_{j=1}^{\infty}1_{F_j}(z)
\int_{\B_n}1_{D\left(z,\frac{\sigma}{2}\right)^c}(w)\phi(z,w)\,dv_{\alpha}(w)\\
 & = & \norm{T_\mu}_{\mathcal{L}(A^p_\alpha, A^p_\alpha)} \sum_{j=1}^{\infty}1_{F_j}(z)\int_{\abs{w}>\delta} \frac{(1-\abs{\varphi_z(w)}^{2})^{-\frac{1}{p}}}{\abs{1-\overline{z}w}^{n+1+\alpha}}\,dv_\alpha(w)\\
 & \leq & \norm{T_\mu}_{\mathcal{L}(A^p_\alpha, A^p_\alpha)} \int_{\abs{w}>\delta} \frac{(1-\abs{w}^2)^{-\frac{1}{p}}(1-\abs{z}^2)^{-\frac{1}{p}}}{\abs{1-\overline{z}w}^{n+1+\alpha-\frac{2}{p}}}\,dv_\alpha(w).
\end{eqnarray*}
Here we have used the change of variable $w'=\varphi_z(w)$ and that the sets $F_j$ are pairwise disjoint.  Pick a number $a=a(n,\alpha,p)$ satisfying
$$
1<a<p\ \textnormal{ and }\  a\left(n+1+\alpha-\frac{1}{p}\right)<n+1+\alpha.
$$
Note that the second condition can be rephrased as $p(n+1+\alpha)<a'$, so it is clear that we can select the number $a$ with the desired properties.  Now apply H\"older's inequality with $\frac{1}{a}+\frac{1}{a'}=1$ to see that
$$
\int_{\abs{w}>\delta} \frac{(1-\abs{w}^2)^{-\frac{1}{p}}}{\abs{1-\overline{z}w}^{n+1+\alpha-\frac{2}{p}}}\,dv_\alpha(w)\leq \left(\int_{\abs{w}>\delta} \frac{(1-\abs{w}^2)^{-\frac{a}{p}}}{\abs{1-\overline{z}w}^{a\left(n+1+\alpha-\frac{2}{p}\right)}}\,dv_\alpha(w)\right)^{\frac{1}{a}} \left(v_\alpha\{w:\abs{w}>\delta\}\right)^{\frac{1}{a'}}.
$$
But, $\left(v_\alpha\{w:\abs{w}>\delta\}\right)^{\frac{1}{a'}}=C(n,\alpha)(1-\delta^{2n})^{\frac{1}{a'}}$,
and by Lemma \ref{Growth} with $t=\alpha-\frac{a}{p}$ and $s=a\left(n+1+\alpha-\frac{2}{p}\right)$, we have
\begin{eqnarray*}
\left(\int_{\abs{w}>\delta} \frac{(1-\abs{w}^2)^{-\frac{a}{p}}}{\abs{1-\overline{z}w}^{a\left(n+1+\alpha-\frac{2}{p}\right)}}\,dv_\alpha(w)\right)^{\frac{1}{a}} & \leq & \left(\int_{\B_n} \frac{(1-\abs{w}^2)^{-\frac{a}{p}}}{\abs{1-\overline{z}w}^{a\left(n+1+\alpha-\frac{2}{p}\right)}}\,dv_\alpha(w)\right)^{\frac{1}{a}}\\
& \leq & \left(\int_{\B_n} \frac{(1-\abs{w}^2)^{\alpha-\frac{a}{p}}}{\abs{1-\overline{z}w}^{a\left(n+1+\alpha-\frac{2}{p}\right)}}
\,dv(w)\right)^{\frac{1}{a}}\leq C(n,\alpha,p),
\end{eqnarray*}
since $a\left(n+1+\alpha-\frac{2}{p}\right)<n+1+\alpha-\frac{a}{p}$ by the choice of $a$.  This then gives
$$
J\lesssim \norm{T_\mu}_{\mathcal{L}(A^p_\alpha, A^p_\alpha)}(1-\delta^{2n})^{\frac{1}{a'}}(1-\abs{z}^2)^{-\frac{1}{p}},
$$
with the restrictions on $a$ giving the corresponding restrictions on $\gamma$ in the statement of the lemma.
\end{proof}

\begin{lm}
\label{Tech2}
Let $1<p<\infty$ and $\mu$ be a $A^p_\alpha$ Carleson measure.  Suppose that $F_j, K_j\subset\B_n$ are Borel sets and $a_j\in L^\infty$ and $b_j\in L^\infty(\B_n;\mu)$ are functions of norm at most 1 for all $j$.  If
\begin{itemize}
\item[(i)] $\beta(F_j, K_j)\geq\sigma\geq 1$;
\item[(ii)] $\textnormal{supp}\, a_j\subset F_j\ $ and $\ \textnormal{supp}\, b_j\subset K_j$;
\item[(iii)] every $z\in\B_n$ belongs to at most $N$ of the sets $F_j$,
\end{itemize}
then $\sum_{j=1}^{\infty} M_{a_j} P_{\mu} M_{b_j}$ is a bounded operator from $A^p_{\alpha}$ to $L^p_{\alpha}$ and there is a function $\beta_{p,\alpha}(\sigma)\to 0$ when $\sigma\to\infty$ such that
\begin{equation}
\label{Tech2-Est1}
\norm{\sum_{j=1}^{\infty} M_{a_j} P_{\mu} M_{b_j}f}_{L^p_{\alpha}}\leq N\beta_{p,\alpha}(\sigma)\norm{T_\mu}_{\mathcal{L}(A^p_\alpha, A^p_\alpha)}\norm{f}_{A^p_{\alpha}},
\end{equation}
and for every $f\in A^p_{\alpha}$
\begin{equation}
\label{Tech2-Est2}
\sum_{j =1}^{\infty}\norm{M_{a_j} P_{\mu} M_{b_j} f}_{L^p_{\alpha}}^p\leq N\beta_{p,\alpha}^p(\sigma)\norm{T_\mu}_{\mathcal{L}(A^p_\alpha, A^p_\alpha)}^p\norm{f}^{p}_{A^p_{\alpha}}.
\end{equation}
\end{lm}

\begin{proof}
Since $\mu$ is a Carleson measure for $A^p_{\alpha}$, $\imath_p: A_{\alpha}^p\to L^p(\B_n;\mu)$ is bounded, with $\norm{\imath_p}\lesssim \norm{\mu}^{\frac{1}{p}}_{\textnormal{RKM}}$ and so it is enough to prove the following two estimates:
\begin{equation}
\label{Tech2-Est1-Red}
\norm{\sum_{j=1}^{\infty} M_{a_j} P_{\mu} M_{b_j}f}_{L^p_\alpha}\leq N\kappa_{p,\alpha}(\delta)\norm{T_\mu}_{\mathcal{L}(A^p_\alpha, A^p_\alpha)}^{1-\frac{1}{p}}\norm{f}_{L^p(\B_n;\mu)},
\end{equation}
and
\begin{equation}
\label{Tech2-Est2-Red}
\sum_{j =1}^{\infty}\norm{M_{a_j} P_{\mu} M_{b_j} f}_{L^p_\alpha}^p\leq N\kappa_{p,\alpha}^p(\delta)\norm{T_\mu}_{\mathcal{L}(A^p_\alpha, A^p_\alpha)}^{p-1}\norm{f}^{p}_{L^p(\B_n;\mu)}
\end{equation}
where $\delta=\tanh\frac{\sigma}{2}$ and $\kappa_{p,\alpha}(\delta)\to 0$ as $\delta\to 1$.  Estimates \eqref{Tech2-Est1-Red} and \eqref{Tech2-Est2-Red} imply \eqref{Tech2-Est1} and \eqref{Tech2-Est2}
via an application of Lemma \ref{CM}.

First, consider the case when $N=1$, and so the sets $\{F_j\}$ are pairwise disjoint.  Set
$$
\Phi(z,w)=\sum_{j=1}^{\infty} 1_{F_j}(z) 1_{K_j}(w) \frac{1}{\abs{1-\overline{z}w}^{n+1+\alpha}}.
$$
Suppose now that $f\in L^p(\B_n;\mu)$, $\norm{a_j}_{L^\infty}$ and $\norm{b_j}_{L^\infty(\B_n;\mu)}\leq 1$.  Easy estimates show
\begin{eqnarray*}
\abs{\sum_{j=1}^{\infty} M_{a_j} P_{\mu} M_{b_j} f(z)} & = & \abs{\sum_{j=1}^{\infty}a_j(z)\int_{\B_n} \frac{b_j(w) f(w)}{(1-\overline{w} z)^{n+1+\alpha}} \,d\mu(w)}\\
& \leq & \int_{\B_n} \Phi(z,w) \abs{f(w)}\,d\mu(w),
\end{eqnarray*}
which implies that it suffices to prove that the operator with kernel $\Phi(z,w)$ is bounded between the necessary spaces.
Set $h(z)=(1-\abs{z}^2)^{-\frac{1}{pq}}$ and observe that Lemma \ref{Tech1} gives
$$
\int_{\B_n} \Phi(z,w) h(w)^q \,d\mu(w)\lesssim \norm{T_\mu}_{\mathcal{L}(A^p_\alpha, A^p_\alpha)} (1-\delta^{2n})^{\gamma}h(z)^{q}.
$$
While Lemma \ref{Growth}, plus a simple computation, implies that
$$
\int_{\B_n}\Phi(z,w) h(z)^{p} \,dv_\alpha(z)\lesssim h(w)^{p}.
$$
Schur's Lemma and Lemma \ref{Schur} then give that the operator with kernel $\Phi(z,w)$ is bounded from $L^p(\B_n;\mu)$ to $L^p_{\alpha}$ with norm controlled by a constant $C(n,\alpha,p)$ times $k_{p,\alpha}(\delta)\norm{T_\mu}_{\mathcal{L}(A^p_\alpha, A^p_\alpha)}^{1-\frac{1}{p}}$.  We thus have \eqref{Tech2-Est1-Red} when $N=1$.  Since the sets $F_j$ are disjoint in this case, then we also have \eqref{Tech2-Est2-Red} because
$$
\sum_{j =1}^{\infty}\norm{M_{a_j} P_{\mu} M_{b_j} f}_{L^p_{\alpha}}^p=\norm{\sum_{j=1}^{\infty} M_{a_j}P_{\mu} M_{b_j} f}_{L^p_{\alpha}}^p.
$$

Now suppose that $N>1$.  Let $z\in \B_n$ and let $S(z)=\left\{j: z\in F_j\right\}$, ordered according to the index $j$.
Each $F_j$ admits a disjoint decomposition $F_j=\bigcup_{k=1}^{N} A_{j}^{k}$ where $A_{j}^{k}$ is the set of $z\in F_j$
such that
$j$ is the $i^{\textnormal{th}}$ element of $S(z)$.
Then, for $1\leq k\leq N$ the sets $\{A_{j}^k: j\geq 1\}$ are pairwise disjoint.
Hence, we can apply the computations obtained above to conclude that
\begin{eqnarray*}
\sum_{j=1}^{\infty} \norm{M_{a_j} P_{\mu} M_{b_j} f}_{L^p_{\alpha}}^p & = & \sum_{j=1}^{\infty} \sum_{k=1}^{N} \norm{M_{a_j 1_{A^k_j}} P_\mu M_{b_j} f}_{L^p_{\alpha}}^p\\
 & = & \sum_{k=1}^{N}\sum_{j=1}^{\infty} \norm{M_{a_j 1_{A^k_j}} P_\mu M_{b_j} f}_{L^p_{\alpha}}^p\\
 & \lesssim & N k_{p,\alpha}^p(\delta)\norm{T_\mu}_{\mathcal{L}(A^p_\alpha, A^p_\alpha)}^{p-1}\norm{f}_{A^p_\alpha}^p.
\end{eqnarray*}
This gives \eqref{Tech2-Est2-Red}, and \eqref{Tech2-Est1-Red} follows from similar computations.
\end{proof}

\begin{lm}
\label{Approx1}
Let $1<p<\infty$ and $\sigma\geq 1$.  Suppose that $a_1,\ldots, a_k\in L^\infty$ are functions of norm at most 1 and that $\mu$
is a $A^p_\alpha$ Carleson measure.  Consider the covering of\/ $\B_n$ given by Lemma \ref{SuaGeo2} for these values of $k$ and $\sigma\geq 1$.  Then there is a positive constant $C(p,k,n,\alpha)$ such that
\begin{equation}
\label{Approx-Est1}
\norm{\left[ \prod_{i=1}^{k} T_{a_i} \right] T_{\mu}-\sum_{j=1}^\infty M_{1_{F_{0,j}}}
\left[ \prod_{i=1}^{k}T_{a_i}  \right] T_{\mu1_{F_{k+1,j}}}}_{\mathcal{L}(A^p_{\alpha},
L^p_{\alpha})}\lesssim \beta_{p,\alpha}(\sigma)\norm{T_{\mu}}_{\mathcal{L}(A^p_{\alpha}, A^p_{\alpha})},
\end{equation}
where $\beta_{p,\alpha}(\sigma)\to 0$ as $\sigma\to\infty$.
\end{lm}
\begin{proof}
We break the proof up into two steps.  We will prove that
\begin{equation}
\label{Step1}
\norm{\left[ \prod_{i=1}^{k}T_{a_i} \right] T_{\mu}-\sum_{j=1}^\infty M_{1_{F_{0,j}}}
\left[ \prod_{i=1}^{k} T_{a_i1_{F_{i,j}}} \right] T_{1_{F_{k+1,j}}\mu}}_{\mathcal{L}(A^p_{\alpha},L^p_{\alpha})} \lesssim  \beta_{p,\alpha}(\sigma)\norm{T_{\mu}}_{\mathcal{L}(A^p_{\alpha}, A^p_{\alpha})}
\end{equation}
and
\begin{equation}
\label{Step2}
\norm{\sum_{j=1}^\infty M_{1_{F_{0,j}}}\left[ \prod_{i=1}^{k}T_{a_i}  \right] T_{\mu1_{F_{k+1,j}}}-
\sum_{j=1}^\infty M_{1_{F_{0,j}}}\left[ \prod_{i=1}^{k} T_{a_i 1_{F_{i,j}}}  \right]
T_{\mu1_{F_{k+1,j}}}}_{\mathcal{L}(A^p_{\alpha}, L^p_{\alpha})}
\lesssim \beta_{p,\alpha}(\sigma)\norm{T_\mu}_{\mathcal{L}(A^p_{\alpha}, A^p_{\alpha})},
\end{equation}
where the constants depend on $p,k,n$, and $\alpha$.
It is obvious that each of these inequalities, when combined give the desired estimate in the Lemma.

For $0\leq m\leq k+1$, define the operators $S_m\in \mathcal{L}\left(A_{\alpha}^p,L_{\alpha}^p\right)$ by
$$
S_m=\sum_{j=1}^{\infty} M_{1_{F_{0,j}}}\left[ \prod_{i=1}^{m} T_{1_{F_{i,j}}a_i}\prod_{i=m+1}^{k} T_{a_i} \right] T_\mu.
$$
Clearly we have
$S_0=\sum_{j=1}^{\infty} M_{1_{F_{0,j}}} \left[\prod_{i=1}^{k} T_{a_i}\right] T_\mu
= \left[\prod_{i=1}^{k} T_{a_i}\right]T_\mu$, with convergence in the strong operator topology.  Similarly, we have
$$
S_{k+1}=\sum_{j=1}^{\infty} M_{1_{F_{0,j}}}\left[\prod_{i=1}^{k} T_{a_i1_{F_{i,j}}} \right] T_{\mu1_{F_{k+1,j}}}.
$$
When $0\leq m\leq k-1$, a simple computation gives that
$$
S_m-S_{m+1}=\sum_{j=1}^{\infty} M_{1_{F_{0,j}}} \left[\prod_{i=1}^{m}T_{1_{F_{i,j}}a_i} \right]
T_{1_{F_{m+1,j}^c}a_{m+1}}  \left[ \prod_{i=m+2}^{k}T_{a_i}\right]  T_\mu.
$$
Here, of course, we should interpret this product as the identity when the lower index is greater than the upper index.  Take any $f\in A^p_{\alpha}$ and apply Lemma \ref{Tech2}, in particular \eqref{Tech2-Est2}, Lemma \ref{SuaGeo2} and some obvious estimates to see that
\begin{eqnarray*}
\norm{\left(S_m-S_{m+1}\right)f}_{L^p_{\alpha}}^p & \leq &
C(p)^{pm}\sum_{j}\norm{M_{1_{F_{m,j}}a_m}P_\alpha M_{1_{F_{m+1,j}^c}a_{m+1}}
\left[\prod_{i=m+2}^k T_{a_i} \right] T_\mu f}_{L^p_{\alpha}}^p\\
& \leq & C(p)^{pm}N\beta_{p,\alpha}^p(\sigma)\norm{\left[\prod_{i=m+2}^k T_{a_i} \right] T_\mu f}_{L^p_{\alpha}}^p\\
& \leq & C(p)^{p(k-1)}N\beta_{p,\alpha}^p(\sigma)\norm{T_\mu}_{\mathcal{L}(A^p_{\alpha},A^p_{\alpha})}^p\norm{f}_{A^p_{\alpha}}^p.
\end{eqnarray*}
Also,
$$
S_{k}-S_{k+1}=\sum_{j=1}^{\infty} M_{1_{F_{0,j}}} \left[ \prod_{i=1}^{k}T_{1_{F_{i,j}}a_i} \right] T_{\mu1_{F_{k+1,j}^c}},
$$
and again applying Lemma \ref{Tech2}, and in particular \eqref{Tech2-Est2}, we find that
$$
\norm{\left(S_k-S_{k+1}\right)f}_{L^p_{\alpha}}^p \leq C_p^{pk}N\beta_{p,\alpha}^p(\sigma)\norm{T_\mu}_{\mathcal{L}(A^p_{\alpha},A^p_{\alpha})}^p\norm{f}^p_{A^p_{\alpha}}.
$$
Since $N=N(n)$, we have the following estimates for $0\leq m\leq k$,
$$
\norm{\left(S_m-S_{m+1}\right)f}_{L^p_{\alpha}}\lesssim \beta_{p,\alpha}(\sigma)\norm{T_\mu}_{\mathcal{L}(A^p_{\alpha},A^p_{\alpha})}\norm{f}_{A^p_{\alpha}}.
$$
But from this it is immediate that \eqref{Step1} holds,
$$
\norm{\left(S_0-S_{k+1}\right)f}_{L^p_{\alpha}}\leq \sum_{m=0}^{k}\norm{\left(S_m-S_{m+1}\right)f}_{L^p_{\alpha}}\lesssim \beta_{p,\alpha}(\sigma)\norm{T_\mu}_{\mathcal{L}(A^p_{\alpha},A^p_{\alpha})}\norm{f}_{A^p_{\alpha}}.
$$

The idea behind \eqref{Step2} is similar.  For $0\leq m\leq k$, define the operator
$$
\tilde{S}_m=\sum_{j=1}^{\infty} M_{1_{F_{0,j}}}\left[\prod_{i=1}^{m} T_{1_{F_{i,j}}a_i}\prod_{i=m+1}^{k} T_{a_i}\right]
T_{\mu1_{F_{k+1,j}}},
$$
so we have
\begin{eqnarray*}
\tilde{S}_0 & = & \sum_{j=1}^{\infty} M_{1_{F_{0,j}}}\left[\prod_{i=1}^{k} T_{a_i}\right] T_{\mu1_{F_{k+1,j}}}\\
\tilde{S}_k & = & \sum_{j=1}^{\infty}M_{1_{F_{0,j}}}\left[\prod_{i=1}^{k}T_{a_i1_{F_{i,j}}}\right]T_{\mu1_{F_{k+1,j}}}.
\end{eqnarray*}
When $0\leq m\leq k-1$, a simple computation gives
$$
\tilde{S}_m-\tilde{S}_{m+1}=\sum_{j=1}^{\infty}M_{1_{F_{0,j}}}\Big[\prod_{i=1}^{m} T_{1_{F_{i,j}}a_i}\Big]
T_{1_{F_{m+1,j}^c}a_{m+1}}  \left[\prod_{i=m+2}^{k}T_{a_i}\right]  T_{\mu1_{F_{k+1,j}}}.
$$
Again, applying obvious estimates and using Lemma \ref{Tech2} one concludes that
\begin{eqnarray*}
\norm{\left(\tilde{S}_m-\tilde{S}_{m+1}\right)f}^p_{L_{\alpha}^p} & \leq & C(p)^{p(k-1)}\beta_{p,\alpha}^{p}(\sigma)\sum_{j=1}^{\infty}\norm{T_{\mu1_{F_{k+1,j}}}f}_{A^p_{\alpha}}^{p}\\
 & \leq & C(p)^{p(k-1)}\beta_{p,\alpha}^{p}(\sigma) \norm{T_{\mu}}_{\mathcal{L}(A^p_{\alpha},A^p_{\alpha})}^{\frac{p}{q}} \sum_{j=1}^{\infty} \norm{1_{F_{k+1,j}} f}_{L^p(\mu)}^p\\
 & \leq & C(p)^{p(k-1)}\beta_{p,\alpha}^{p}(\sigma) \norm{T_{\mu}}_{\mathcal{L}(A^p_{\alpha},A^p_{\alpha})}^{\frac{p}{q}} \norm{f}_{L^p(\mu)}^p\\
 & \leq & NC(p)^{p(k-1)}\beta_{p,\alpha}^{p}(\sigma)
 \norm{T_{\mu}}_{\mathcal{L}(A^p_{\alpha},A^p_{\alpha})}^{\frac{p}{q}+1}  \norm{f}_{A^p_{\alpha}}^p.
\end{eqnarray*}
Here the second inequality uses Lemma \ref{CM-Cor}, the next inequality uses that the sets $\{F_{k+1,j}\}$
form a covering of $\B_n$ with at most $N=N(n)$ overlap, and the last inequality uses Lemma \ref{CM}.
Summing up, for $0\leq m\leq k-1$ we have
$$
\norm{\left(\tilde{S}_m-\tilde{S}_{m+1}\right)f}_{L_{\alpha}^p}\lesssim \beta_{p,\alpha}(\sigma)\norm{T_{\mu}}_{\mathcal{L}(A^p_{\alpha},A^p_{\alpha})}\norm{f}_{A^p_{\alpha}},
$$
which implies
$$
\norm{\left(\tilde{S}_0-\tilde{S}_{k}\right)f}_{L_{\alpha}^p}\leq \sum_{m=0}^{k-1}\norm{\left(\tilde{S}_m-\tilde{S}_{m+1}\right)f}_{L_{\alpha}^p}\lesssim \beta_{p,\alpha}(\sigma)\norm{T_{\mu}}_{\mathcal{L}(A^p_{\alpha},A^p_{\alpha})}\norm{f}_{A^p_{\alpha}},
$$
giving \eqref{Step2}.
\end{proof}

\begin{lm}
\label{Approx2}
Let
$$
S=\sum_{i=1}^{m} \left[\prod_{l=1}^{k_i} T_{a^i_l} \right] T_{\mu_i},
$$
where $a^i_j\in L^\infty$, $k_1,\ldots, k_m\leq k$ and $\mu_i$ are complex-valued measures on $\B_n$ such that $\abs{\mu_i}$
are $A^p_\alpha$ Carleson measures.  Given $\epsilon>0$, there is $\sigma=\sigma(S,\epsilon)\geq 1$ such that if $\{F_{i,j}\}_{j=1}^{\infty}$ and $0\leq i\leq k+1$ are the sets given by Lemma \ref{SuaGeo2} for these values of $\sigma$ and $k$, then
$$
\norm{S-\sum_{j=1}^{\infty}M_{1_{F_{0,j}}}\sum_{i=1}^{m} \left[\prod_{l=1}^{k_i} T_{a^i_l} \right] T_{\mu_i1_{F_{k+1,j}}}}_{\mathcal{L}(A^p_{\alpha},L^p_{\alpha})}<\epsilon.
$$
\end{lm}
\begin{proof}
First, suppose that $\mu_i$ are non-negative measures.
It suffices to prove the result in this situation since for general $\mu_i$ we can decompose
$\mu_i=\mu_{i,1}-\mu_{i,2}+i\mu_{i,3}-i\mu_{i,4}$, where each $\mu_{i,j}$ is a non-negative measure, and hence a Carleson measure.

Without loss of generality, set $k_i=k$ for all $i=1,\ldots, m$.  This can be accomplished by placing copies of the identity in each product if necessary.  We now apply Lemma \ref{Approx1} to each term in the operator $S$.  By Lemma \ref{Approx1}, for $\sigma=\sigma(S,\epsilon)$ sufficiently large we have
$$
\norm{\left[\prod_{j=1}^{k} T_{a^i_j} \right] T_{\mu_i}-
\sum_{j=1}^{\infty}M_{1_{F_{0,j}}}\left[\prod_{l=1}^{k} T_{a^i_l} \right] T_{\mu_i1_{F_{k+1,j}}}}_{\mathcal{L}(A^p_{\alpha},L^p_{\alpha})}<\frac{\epsilon}{m}
$$
for $i=1,\ldots, m$.  Then, summing this estimate we see that
$$
\norm{S-\sum_{i=1}^{m}\sum_{j=1}^{\infty}M_{1_{F_{0,j}}} \Big[\prod_{l=1}^{k_i} T_{a^i_l} \Big] T_{\mu_i1_{F_{k+1,j}}}}_{\mathcal{L}(A^p_{\alpha},L^p_{\alpha})}<\epsilon.
$$
But, for every $i=1,\ldots, m$ we have that
$\sum_{j=1}^{\infty}M_{1_{F_{0,j}}} \left[\prod_{j=1}^{k_i} T_{a^i_j}\right] T_{1_{F_{k+1,j}}\mu_i}$ converges in the strong operator topology,  and so
$$
\norm{S-\sum_{j=1}^{\infty}M_{1_{F_{0,j}}}\sum_{i=1}^{m} \left[\prod_{l=1}^{k_i} T_{a^i_l}\right] T_{1_{F_{k+1,j}}\mu_i}}_{\mathcal{L}(A^p_{\alpha},L^p_{\alpha})}<\epsilon,
$$
as desired.
\end{proof}

We are finally ready to prove Theorem~\ref{Approx3}.

\begin{proof}
Since $S\in \mathcal{T}_{p,\alpha}$, there is $S_0=\sum_{i=1}^{m}\prod_{l=1}^{k_i}T_{a_l^i}$ such that
$$
\norm{S-S_0}_{\mathcal{L}(A^p_{\alpha}, A^p_{\alpha})}<\epsilon,
$$
where $a_l^i\in L^\infty$ and $k_i$ are positive integers.
Set $k=\max\left\{k_i:i=1,\ldots,m\right\}$.  By Lemma \ref{Approx2} we can choose $\sigma=\sigma(S_0,\epsilon)$ and
sets $F_j=F_{0,j}$ and $G_j=F_{k+1,j}$ with
$$
\norm{S_0T_\mu-\sum_{j=1}^{\infty}M_{1_{F_{j}}} S_0 T_{\mu 1_{G_{j}}}}_{\mathcal{L}(A^p_{\alpha},L^p_{\alpha})}
<\epsilon.
$$
We have that (i), (ii), (iii) and (iv) hold by Lemma \ref{SuaGeo2}.
Now, for $f\in A^p_{\alpha}$ we have
\begin{eqnarray*}
\norm{\sum_{j=1}^{\infty} M_{1_{F_j}}\left(S-S_0\right)T_{\mu 1_{G_j}}f}_{L^p_{\alpha}}^p & = & \sum_{j=1}^{\infty}\norm{ M_{1_{F_j}}\left(S-S_0\right)T_{\mu 1_{G_j}}f}_{L^p_{\alpha}}^p\\
& \leq & \epsilon^p\sum_{j=1}^{\infty}\norm{T_{\mu 1_{G_j}} f}_{A^p_{\alpha}}^p\\
& \leq & \epsilon^p\sum_{j=1}^{\infty} \norm{1_{G_j}f}_{L^p(\mu)}^p\\
& \leq & N\epsilon^p\norm{T_\mu}_{\mathcal{L}(A^p_{\alpha}, A^p_{\alpha})}^p\norm{f}_{A^p_{\alpha}}^p.
\end{eqnarray*}
Therefore, the triangle inequality gives

\begin{eqnarray*}
\norm{ST_\mu-\sum_{j=1}^{\infty} M_{1_{F_j}}ST_{\mu 1_{G_j}}}_{\mathcal{L}(A^p_{\alpha}, L^p_{\alpha})} &\leq&
\norm{S-S_0}_{\mathcal{L}(A^p_{\alpha}, A^p_{\alpha})} \norm{T_\mu}_{\mathcal{L}(A^p_{\alpha}, A^p_{\alpha})} + \epsilon \\
&+&
\norm{\sum_{j=1}^{\infty} M_{1_{F_j}}\left(S-S_0\right)T_{\mu 1_{G_j}}}_{\mathcal{L}(A^p_{\alpha}, L^p_{\alpha})}
\lesssim \epsilon ,
\end{eqnarray*}
where the constant of the last inequality depends only on $\norm{T_\mu}_{\mathcal{L}(A^p_{\alpha}, A^p_{\alpha})}$ and $n$.
\end{proof}

\section{A uniform Algebra and Its Maximal Ideal Space}
\label{UniformAlg}

We consider the algebra $\mathcal{A}$ of all bounded functions that are uniformly continuous from the metric space $(\B_n,\rho)$
into the metric space $(\C,\abs{\,\cdot\,})$.  We then associate to $\mathcal{A}$ its maximal ideal space $M_{\mathcal{A}}$,
which is the set of all non-zero multiplicative linear functionals from $\mathcal{A}$ to $\C$.
Endowed with the weak-star topology, this is then a compact Hausdorff space.
Via the Gelfand transform we can view the elements of $\mathcal{A}$ as continuous functions on $M_{\mathcal{A}}$
as given by $\hat{a}(f)=f(a)$, where $f$ is a multiplicative linear functional.  Since $\mathcal{A}$ is a commutative $C^*$ algebra, the Gelfand transform
is an isomorphism.  It is also obvious that point evaluation is a multiplicative linear functional, and so
$\B_n\subset M_{\mathcal{A}}$.  Moreover, since $\mathcal{A}$ is a $C^*$ algebra, $\B_n$ is dense in $M_{\mathcal{A}}$.
Also, one can easily see that the Euclidean topology on $\B_n$ agrees with the topology induced by $M_{\mathcal{A}}$.

We next state several lemmas and facts that will be useful going forward.  Their proof can be found in \cite{Sua}.
For a set $E\subset M_{\mathcal{A}}$, the closure of $E$ in the space $M_{\mathcal{A}}$ is denoted $\overline{E}$.
Note that if $E\subset r\B_n$, where $0<r<1$, then this closure is the same as the Euclidean closure.

\begin{lm}
\label{Invariance}
Let $z,w,\xi\in\B_n$.  Then there is a positive constant $C(n)$ such that
$$
\rho(\varphi_z(\xi),\varphi_w(\xi))\lesssim \frac{\rho(z,w)}{1-\abs{\xi}^2}.
$$
\end{lm}

\begin{lm}
\label{Extend}
Let $(E,d)$ be a metric space and $f:\B_n\to E$ be a continuous map.  Then $f$ admits a continuous extension from $M_{\mathcal{A}}$ into $E$ if and only if $f$ is $(\rho,d)$ uniformly continuous  and $\overline{f(\B_n)}$ is compact.
\end{lm}

Let $x\in M_{\mathcal{A}}$ and suppose that $\{z_\omega\}$ is a net in $\B_n$ converging to $x$.
By compactness, the net $\{\varphi_{z_s}\}$ in the product space $M^{\B_n}_{\mathcal{A}}$ admits a convergent subnet
$\{\varphi_{z_{\omega_{\tau}}}\}$.
That is, there is a function $\varphi:\B_n\to M_{\mathcal{A}}$ such that $f\circ \varphi_{z_{\omega_{\tau}}}\to f\circ\varphi$.  Moreover, it can be shown that the whole net $\{z_\omega\}$ converges to $x$, and that $\varphi$ does not depend on the net.  We then denote the limit by $\varphi_x$ and one can easily observe that $\varphi_x(0)=x$.  This gives the following Lemma.

\begin{lm}
\label{Maximal3}
Let $\{z_\alpha\}$ be a net in $\B_n$ converging to $x\in M_{\mathcal{A}}$.  Then
\begin{itemize}
\item[(i)] $a\circ\varphi_x\in\mathcal{A}$ for every $a\in\mathcal{A}$.  In particular, $\varphi_x:\B_n\to M_{\mathcal{A}}$ is continuous;
\item[(ii)] $a\circ\varphi_{z_{\omega}}\to a\circ\varphi_x$ uniformly on compact sets of $\B_n$ for every $a\in\mathcal{A}$.
\end{itemize}
\end{lm}

\subsection{Carleson Measures and Approximation}

Given a complex-valued measure $\mu$ whose variation is Carleson,
our next goal is to construct a sequence of functions $B_k(\mu) \in \mathcal{A}$ such that $T_{B_k(\mu)} \rightarrow T_\mu$ in the norm of
$\mathcal{L}(A_\alpha^p, A_\alpha^p)$ for any $\alpha >-1$ and $1<p<\infty$.
As a consequence, we have the following Theorem.
\begin{thm}
\label{Sua2Thm}
The Toeplitz algebra $\mathcal{T}_{p,\alpha}$ equals the closed algebra generated by $\{T_a:a\in\mathcal{A}\}$.
\end{thm}
For $\alpha=0$ this was proved in \cite{Sua}*{Thm.$\,$7.3}.  Here we give a more direct and quantitative proof.
We remark that an additional proof can be given by building on the ideas in \cites{Sua, Sua2}.
Recall that for $\alpha > -1$, 
$$
dv_\alpha(z) := c_\alpha  \, (1-|z|^2)^\alpha \, dv(z),
\ \mbox{ with }\
c_\alpha := \frac{\Gamma(n+\alpha+1)}{n! \,\Gamma(\alpha+1)},
$$
is a probability measure on $\B_n$. Let $\mu$ be a complex-valued, Borel, regular measure on $\B_n$ of finite total variation.  If $z\in\B_n$, $k\ge \alpha$, and $\alpha>-1$, the $(k,\alpha)$-Berezin transform of $\mu$ is the function
$$
B_k(\mu)(z) := \frac{c_k }{c_\alpha }  \, \int_{\B_n} \frac{
(1-|\varphi_z(w)|^2)^{n+1+k} }{ (1-\abs{w}^2)^{n+1+\alpha} }  \, d\mu(w). \\
$$

When $\mu = a\, dv_\alpha$, with $a\in L^1_\alpha$, a change of variables gives $B_k(a\,dv_\alpha) = \int_{\B_n} (a\circ\varphi_z)\,  dv_k$.  Also, observe that if $|\mu|$ is a Carleson measure then $B_0 (\mu) = B(T_\mu)$, the Berezin transform of $T_\mu$.  Finally, observe that for $k\ge \alpha$,
\begin{align*}
|B_k(\mu)(z)| &\le \frac{c_k }{c_\alpha }  \, \int_{\B_n}
\frac{(1-|\varphi_z(w)|^2)^{n+1+\alpha} }{ (1-|w|^2)^{n+1+\alpha} } \, 
d|\mu|(w)
=  \frac{c_k }{c_\alpha }  \,  B_\alpha(|\mu|)(z).
\end{align*}

\begin{lm}
\label{bo-lip}
If $|\mu|$ is a $A^p_\alpha$ Carleson measure and $k\ge \alpha$,
then $B_k(\mu)$ is a bounded Lipschitz function from $(\B_n, \rho)$
into $(\mathbb{C}, \abs{\,\cdot\,})$. Specifically, there are constants depending on $n, k$ and $\alpha$ such that
$$
|B_k(\mu)(z)| \lesssim  \,  B_\alpha(|\mu|)(z) \ \ \mbox{ and }\ \
|B_{k}(\mu) (z_1) - B_{k}(\mu) (z_2)| \lesssim    \,
\norm{B_{\alpha}(|\mu|)}_\infty  \, \rho(z_1,z_2).
$$
\end{lm}

\begin{proof}
The first estimate is immediate from the definitions, and so we turn to the second.  Let $z_1,z_2\in \B_n$, and assume first that $w\in\B_n$ is such that $\abs{\varphi_{z_1}(w)} \leq  \abs{\varphi_{z_2}(w)}$.  Applying Lagrange's Theorem to the function $f(x) = x^s$ for $0\le x\le 1$ and $s >1$ yields
\begin{align*}
|(1-|\varphi_{z_1}(w)|^2)^{s} - (1-|\varphi_{z_2}(w)|^2)^{s}|
&\le 2s \, (1-|\varphi_{z_1}(w)|^2)^{s-1}
 \left| |\varphi_{z_1}(w)| -  |\varphi_{z_2}(w)|    \right| \\*[1mm]
&\hspace{-40mm}\le  2s \, (1-|\varphi_{z_1}(w)|^2)^{s-1}
\rho(z_1,z_2)\left(1-|\varphi_{z_1}(w)| \, |\varphi_{z_2}(w)| \right)   \\*[1mm]
&\hspace{-40mm}\le  2s \, \left(1-|\varphi_{z_1}(w)|^2\right)^{s}  \rho(z_1,z_2) ,
\end{align*}
where the inequality in the middle uses a well-known stronger version of the triangle inequality for the metric $\rho$.
If $w\in\B_n$ is such that $\abs{\varphi_{z_2}(w)} \leq\abs{\varphi_{z_1}(w)}$ we get a symmetric inequality.
Thus, for $k\ge \alpha$, taking $s = n+1+k$, there are constants depending on $k, n$ and $\alpha$  such that
$$
|B_{k}(\mu) (z_1) - B_{k}(\mu) (z_2)| \lesssim  \big[ B_{k}(|\mu|) (z_1)
+  B_{k}(|\mu|) (z_2) \big] \, \rho(z_1,z_2) \lesssim   \,
\|B_{\alpha}(|\mu|)\|_\infty  \, \rho(z_1,z_2).
$$
\end{proof}

For the next lemma, we will need truncated versions of $B_k$ and a corresponding adjoint.   To define these operators, if $0<r< 1$, let
$$
B_{k,r}(\mu)(z) :=\frac{c_k }{c_\alpha }   \, \int_{|\varphi_z(w)|<r}
\frac{ (1-|\varphi_z(w)|^2)^{n+1+k} }{ (1-|w|^2)^{n+1+\alpha} }  \, d\mu(w)
$$
and for $h\in L^1_\alpha$,
$$
B_{k,r}^\ast(h)(w) := \frac{c_k }{c_\alpha }   \,
\int_{|\varphi_w(z)|<r} \frac{ (1-|\varphi_z(w)|^2)^{n+1+k} }{ (1-|w|^2)^{n+1+\alpha} }   \, h(z)  \,dv_\alpha(z).
$$
Since the measure $\frac{dv_\alpha(w)}{(1-|w|^2)^{n+1+\alpha}}$ is conformally invariant and $dv_k= c_k (1-|u|^2)^k \,dv$ is a probability measure,
it follows immediately that $B_{k,r}^\ast$ is a contraction on $L^1_\alpha$.
The change of variables $\varphi_w(z)=u$ leads to
\begin{equation}
\label{adju}
B_{k,r}^\ast (h)(w) =  \int_{|u|<r}  \frac{ (1-|u|^2)^{n+1+\alpha} }{ |1-\overline{u}w|^{2(n+1+\alpha)} }   \, h(\varphi_w(u))  \,dv_k(u),
\end{equation}
which also shows that $B_{k,r}^\ast$ acts on $L^\infty $ with norm bounded by some positive constant $C(\alpha, n, r)$.
Furthermore, if $|\mu|$ is a $A^p_\alpha$ Carleson and $h\in L^1_\alpha$, Fubini's theorem yields
\begin{equation}\label{bkdual}
\int_{\B_n} B_{k,r} (\mu)(z) h(z) \,dv_\alpha(z) = \int_{\B_n}
B_{k,r}^\ast (h)(w) \,  d\mu(w).
\end{equation}
Let $f$ be a complex-valued $C^1$ function on $\B_n$. Recall that the gradient of $f$ is defined as
$$
\nabla f = \left(\frac{\partial f}{\partial z_1}, \ldots , \frac{\partial f}{\partial z_n},
\frac{\partial f}{\partial \overline{z}_1}, \ldots , \frac{\partial f}{\partial \overline{z}_n}\right)
$$
and the invariant gradient as
$$
\widetilde{\nabla} f (z) = \nabla \left(f\circ \varphi_z\right) (0).
$$
In \cite{Zhu}*{pp.$\,$49} it is shown that if $f$ is holomorphic on $\B_n$ and $\varphi\in \mbox{Aut} (\B_n)$, then
\begin{equation}\label{zzhu}
\abs{\widetilde{\nabla} \left(f\circ \varphi\right) (z)} = \abs{(\widetilde{\nabla} f)\circ \varphi (z)}
\ \ \mbox{ and }\ \
(1-|z|^2)\abs{\nabla f(z)} \leq \abs{\widetilde{\nabla} f (z)} .\\
\end{equation}

\begin{lm}\label{new1}
Let $f\in A_\alpha^p$ and $g\in A^q_\alpha$, with $\frac{1}{p}+\frac{1}{q}=1$, and write
$h=f\overline{g}$. If $r\leq \frac{1}{4}$, there is a positive constant $C(k)$ independent of $r$ such that $C(k) \rightarrow 0$ when
$k\rightarrow \infty$, and
\begin{equation}\label{bkast}
 B_{k,r}^\ast (|h-h(w)|)(w)  \leq
C(k) \int_{\B_n} \left(|g(\zeta) | \, |\widetilde{\nabla} f(\zeta)|+ |f(\zeta) | \, |\widetilde{\nabla} g(\zeta)| \right)\,
\frac{ (1-|\zeta|^2)^{n+1+\alpha} }{ |1-\overline{\zeta}w |^{2(n+1+\alpha)} }\,dv_\alpha(\zeta).
\end{equation}
\end{lm}
\begin{proof}
By \eqref{adju}
\begin{eqnarray}\label{bkast3}
 B_{k,r}^\ast \left(\abs{h-h(w)}\right)(w)
& = & \int_{|u|<r}  \frac{ (1-|u|^2)^{n+1+\alpha} }{ |1- \overline{u}w|^{2(n+1+\alpha)} }  |h(\varphi_w(u))-h(w)| \, dv_k(u) \nonumber \\*[1mm]
& \leq &  2^{n+1+\alpha}  \int_{|u|<r}     |h(\varphi_w(u))-h(w)| \, dv_k(u).
\end{eqnarray}
Furthermore, there is a constant $C(n)$ such that
\begin{eqnarray}\label{rba}
|h(\varphi_w(u))-h(w)| &\lesssim &   |u| \, \sup_{|\xi| \le r}  |\nabla (h\circ\varphi_w)(\xi)| \nonumber \\
&\hspace{-55mm}\lesssim & \hspace{-28mm} |u|  \left[  \sup_{|\xi| \le \frac{1}{4}}\  | \overline{(g\circ\varphi_w)}\nabla (f\circ\varphi_w) | +
 \sup_{|\xi|\le \frac{1}{4}}\   | (f\circ\varphi_w)\overline{\nabla (g\circ\varphi_w)} |
\right] .
\end{eqnarray}
Since the functions $(g\circ\varphi_w)\nabla (f\circ\varphi_w)$ and
$(f\circ\varphi_w)\nabla (g\circ\varphi_w)$  are analytic from
$\B_n$ into $\mathbb{C}^{2n}$, each coordinate is subharmonic, and consequently so is the sum of their absolute values. Furthermore, since the $\ell^1$ and $\ell^2$
norms are equivalent on $\mathbb{C}^{2n}$ via constants depending only on $n$, we see that for some constant $C(n)$ and $|\xi|\le \frac{1}{4}$,
\begin{align*}
|(g\circ\varphi_w) \nabla (f\circ\varphi_w) |\, (\xi)  &\lesssim
4^{2n} \int_{|\varsigma-\xi| < \frac{1}{4}} |(g\circ\varphi_w) \nabla (f\circ\varphi_w) | \, dv(\varsigma)\\
&\lesssim
\int_{|\varsigma| < \frac{1}{2}}
\left[ \frac{ 1-|\varsigma|^2 }{ 1-(\frac{1}{4}) }  \right]^{\alpha+1} |(g\circ\varphi_w) \nabla (f\circ\varphi_w) | \, dv(\varsigma)\\*[1mm]
&=  \frac{C(n)}{c_\alpha } \Big(\frac{ 4 }{ 3 }\Big)^{\alpha+1} \int_{|\varsigma| < \frac{1}{2}}  (1-|\varsigma|^2)   |(g\circ\varphi_w) \nabla (f\circ\varphi_w) |
  \, dv_\alpha(\varsigma).
\end{align*}
In the last integral, using that by \eqref{zzhu},
$(1-|\varsigma|^2)   \abs{\nabla (f\circ\varphi_w)  (\varsigma)} \le 
\abs{(\widetilde{\nabla} f) (\varphi_w (\varsigma))}$, we see that there is a constant $C(n,\alpha)$ such that
\begin{eqnarray*}
|(g\circ\varphi_w)\nabla (f\circ\varphi_w) |\, (\xi)  &\lesssim &
 \int_{|\varsigma| < \frac{1}{2}} |(g\circ\varphi_w)(\varsigma)| \,  |(\widetilde{\nabla} f) (\varphi_w (\varsigma)) | \, dv_\alpha(\varsigma)  \\
&=&  \int_{|\varphi_w(\zeta)|<\frac{1}{2}} |g(\zeta) | \, |\widetilde{\nabla} f(\zeta)| \,
\frac{ (1-|w|^2)^{n+1+\alpha} }{ |1- \overline{\zeta}w|^{2(n+1+\alpha)} }\,dv_\alpha(\zeta)  \\
&\lesssim &  \int_{|\varphi_w(\zeta)|<\frac{1}{2}} |g(\zeta) | \, |\widetilde{\nabla} f(\zeta)| \,
\frac{ (1-|\zeta|^2)^{n+1+\alpha} }{ |1-\overline{\zeta}w|^{2(n+1+\alpha)} }\,dv_\alpha(\zeta),  \\
\end{eqnarray*}
where the equality comes from the change of variable $\zeta = \varphi_w(\varsigma)$ and the last inequality
holds because $(1-|w|^2)\le 8(1-|\zeta|^2)$ when ${|\varphi_w(\zeta)|<\frac{1}{2}}$.
Since the same estimate holds when interchanging $f$ and $g$ in \eqref{rba}, we obtain
$$
|h(\varphi_w(u))-h(w)| \lesssim  |u|  \int_{\B_n} \left(|g(\zeta) | \, |\widetilde{\nabla}
f(\zeta)|+ |f(\zeta) | \, |\widetilde{\nabla} g(\zeta)| \right)\, \frac{ (1-|\zeta|^2)^{n+1+\alpha} }{|1-\overline{\zeta}w|^{2(n+1+\alpha)} }\,dv_\alpha(\zeta),
$$
where the implied constant depends on $n$ and $\alpha$.  Note that
$$
\int_{|u|<r} \!\! |u|\, dv_k(u)\leq C(k) = \int_{\B_n} |u| \,dv_k(u) \to 0
$$
as $k\to \infty$.  Inserting this inequality in \eqref{bkast3}, we have another positive constant
$C(n,\alpha)$ such that
\begin{eqnarray*}
 B_{k,r}^\ast (|h-h(w)|)(w)
& \lesssim & C(k) \int_{\B_n} \left(|g(\zeta) | \, |\widetilde{\nabla} f(\zeta)|+ |f(\zeta) | \,
|\widetilde{\nabla} g(\zeta)| \right)\, \frac{ (1-|\zeta|^2)^{n+1+\alpha} }{ |1-\overline{\zeta}w|^{2(n+1+\alpha)} }\,dv_\alpha(\zeta)
\end{eqnarray*}
which proves the Lemma.
\end{proof}

\begin{thm}
\label{new2}
Let $1< p< \infty$, $\alpha>-1$,  $\abs{\mu}$ a $A^p_\alpha$ Carleson measure and $h= f\overline{g}$,
with $f\in A_\alpha^p$ and $g\in A^q_\alpha$. Then
\begin{equation}\label{bkth}
\left|\int_{\B_n} B_k(\mu)(z) h(z) \,dv_\alpha(z) - \int_{\B_n} h(z) \,d\mu(z) \right|
\leq C(k) \norm{B_\alpha(|\mu|)}_\infty \norm{f}_{A^p_\alpha} \norm{g}_{A^q_\alpha} ,
\end{equation}
where $C(k) \rightarrow 0$ as $k\rightarrow \infty$. In particular, when $1<p<\infty$ and $\alpha>-1$, $T_{B_k(\mu)}\rightarrow T_\mu$ when $k\rightarrow \infty$.
\end{thm}

\begin{proof}
Fix any $0<r\leq \frac{1}{4}$ and split $B_k(\mu) = B_{k,r} (\mu) + E_{k,r} (\mu)$, where
$$
E_{k,r} (\mu)(z) = \frac{c_k}{c_\alpha}   \, \int_{|\varphi_z(w)|\ge r}
\frac{ (1-|\varphi_z(w)|^2)^{n+1+k} }{ (1-|w|^2)^{n+1+\alpha} }  \,
\,d\mu(w).
$$
From \eqref{bkdual} we see that
\begin{eqnarray*}
J_k := \left|\int_{\B_n} B_k(\mu) h \,dv_\alpha - \int_{\B_n} h \,d\mu \right| &\leq &
\left|\int_{\B_n} B_{k,r}(\mu) h \,dv_\alpha - \int_{\B_n} h \,d\mu \right| + \int_{\B_n} \abs{E_{k,r}(\mu) h}\, dv_\alpha \\
&\leq & \int_{\B_n}  \abs{B_{k,r}^\ast (h) -h}\,  d|\mu| +     \int_{\B_n}
E_{k,r}(\abs{\mu}) \,\abs{h}\, dv_\alpha  ,
\end{eqnarray*}
and since
\begin{eqnarray*}
\abs{B_{k,r}^\ast (h) -  h} & \leq &
\abs{B_{k,r}^\ast (h) - B_{k,r}^\ast (1) h}  +  \abs{B_{k,r}^\ast (1)- 1} \, \abs{h} \\
& \leq & B_{k,r}^\ast (\abs{h-h(w)})(w)  +  \abs{B_{k,r}^\ast (1)- 1} \,
\abs{h},
\end{eqnarray*}
we get
\begin{equation}\label{bkmain}
J_k \leq \int_{\B_n} B_{k,r}^\ast (\abs{h-h(w)})d\abs{\mu}(w) + \int_{\B_n} \abs{B_{k,r}^\ast
(1)- 1} \, \abs{h} d\abs{\mu} +  \int_{\B_n} E_{k,r}(\abs{\mu}) \,\abs{h}\, dv_\alpha  .
\end{equation}
So, it is enough to show that each one of the above integrals admits
a bound as in \eqref{bkth}.

For the last integral, observe
\begin{align*}
E_{k,r} (|\mu|)(z) &= \frac{c_k}{c_\alpha}   \,
\int_{|\varphi_z(w)|\ge r}
\frac{ \left(1-|\varphi_z(w)|^2\right)^{n+1+\alpha} }{ \left(1-|w|^2\right)^{n+1+\alpha} }  \left(1-|\varphi_z(w)|^2\right)^{k-\alpha}\, d|\mu|(w) \\
&\leq \frac{c_k}{c_\alpha}   \,\left( 1-r^2 \right)^{k-\alpha} \|B_\alpha(|\mu|)\|_\infty \\*[1mm]
&= \frac{\Gamma(n+k+1)}{\Gamma(k+1)}
\frac{\Gamma(\alpha+1)}{\Gamma(n+\alpha+1)}    \,\left( 1-r^2 \right)^{k-\alpha} \|B_\alpha(|\mu|)\|_\infty \\
&\leq C(\alpha , n)\, k^n    \,\left( 1-r^2 \right)^{k-\alpha} \|B_\alpha(|\mu|)\|_\infty ,
\end{align*}
where the last inequality follows from Stirling's formula. Thus, for
any $0<r<1$ there is a constant $C_r(k) \to 0$ as $k\to\infty$ such
that $E_{k,r} (|\mu|)\le C_r(k)\|B_\alpha(|\mu|)\|_\infty$.

In order to estimate the second integral in \eqref{bkmain}, observe
that since $v_k$ is a probability measure whose mass tends to
accumulate at the origin when $k\to\infty$, then $v_k\left(\B_n\setminus
r\B_n\right)\to 0$ when $k\to\infty$ for any $0<r<1$. Furthermore, since
$$
\abs{B_{k,r}^\ast (1)- 1}  \leq  \abs{B_{k,r}^\ast (1)- v_k(r\B_n)} + v_k\left(\B_n\setminus r\B_n\right),
$$
we only need to show that the the first term of the above sum is
bounded by a constant that  tends to 0 as $k\to\infty$.
Indeed, applying Lagrange's Theorem to the function $f(x) = x^{n+1+\alpha}$ on $[0,1]$, and using that $r\le \frac{1}{4}$,  we find a constant $C(n,\alpha)$ so that
\begin{align*}
\abs{B_{k,r}^\ast (1)(w) - v_k(r\B_n)}  &\leq \int_{|u|<r}
\left|  \frac{ (1-|u|^2)^{n+1+\alpha} }{ |1-\overline{u}w|^{2(n+1+\alpha)} } -1   \right|    \,dv_k(u) \\*[1mm]
&\lesssim     \int_{\B_n}  \Big|     (1-|u|^2)   - |1-\overline{u}w|^2    \Big|     \, dv_k(u)   \\*[1mm]
&\lesssim    \int_{\B_n} |u|    \, dv_k(u) \rightarrow 0
\end{align*}
when $k\rightarrow \infty$. Then $|B_{k,r}^\ast (1)- 1| \le C_2(k) \to 0$ as $k\rightarrow \infty$.

Finally, we use Lemma \ref{new1} to estimate the first integral in \eqref{bkmain}. By \eqref{bkast},
\begin{eqnarray*}
\int_{\B_n} B_{k,r}^\ast \left(|h-h(w)|\right)d|\mu|(w)
& \leq &  C(k) \int_{\B_n} \left(|g(\zeta) | \, |\widetilde{\nabla} f(\zeta)|+ |f(\zeta) | \, |\widetilde{\nabla} g(\zeta)| \right)  B_\alpha (|\mu|)(\zeta)  \, dv_\alpha(\zeta)\\
&\le &  C(k) \| B_\alpha (|\mu|)\|_\infty  \, \left(  \|g \|_{A^q_\alpha} \, \|\widetilde{\nabla} f\|_{L^p_\alpha} + \|f \|_{A^p_\alpha} \, \|\widetilde{\nabla} g\|_{L^q_\alpha}  \right).
\end{eqnarray*}
Since \cite{Zhu}*{Theorem 2.16} says that $\|\widetilde{\nabla} f\|_{L^p_\alpha} \lesssim  \|f\|_{A^p_\alpha}$ and
$\|\widetilde{\nabla} g\|_{L^q_\alpha} \lesssim \|g\|_{A^q_\alpha}$, and Lemma \ref{new1} says that
$C(k) \to 0$ when $k\to\infty$, the theorem follows.

\end{proof}


\subsection{Maps from \texorpdfstring{$M_{\mathcal{A}}$}{the maximal ideal space}
into \texorpdfstring{$\mathcal{L}(A^p_{\alpha}, A^p_{\alpha})$, }{the bounded operators on the weighted Bergman space}}

Define a map
$$
U^{(p,\alpha)}_z f(w):= f(\varphi_z(w))\frac{(1-\abs{z}^2)^{\frac{n+1+\alpha}{p}}}{(1-w\overline{z})^{\frac{2(n+1+\alpha)}{p}}}
$$
where the argument of $(1-w\overline{z})$ is used to define the root appearing above.
A standard change of variable and straightforward computations give
$$
\norm{U^{(p,\alpha)}_z f}_{A^p_{\alpha}}=\norm{f}_{A^p_{\alpha}} \quad\forall f\in A^p_{\alpha},
$$
and $\,U^{(p,\alpha)}_zU^{(p,\alpha)}_z=Id_{A^p_\alpha}$.  For a real number $r$, set
$$
J^r_z(w)=\frac{(1-\abs{z}^2)^{r\frac{n+1+\alpha}{2}}}{(1-w\overline{z})^{r(n+1+\alpha)}}.
$$
Observe that
$$
U^{(p,\alpha)}_z f(w)=f(\varphi_z(w))J_z^{\frac{2}{p}}(w)\quad\textnormal{ and }\quad  U^{(p,\alpha)}_z= T_{J_z^{\frac{2}{p}-1}}U^{(2,\alpha)}_z=U^{(2,\alpha)}_z T_{J_z^{1-\frac{2}{p}}}.
$$
So, if $q$ is the conjugate exponent of $p$, we have
$$
\left(U^{(q,\alpha)}_z\right)^{*}=U^{(2,\alpha)}_z T_{\overline{J_z}^{\frac{2}{q}-1}}
=T_{\overline{J_z}^{1-\frac{2}{q}}}U^{(2,\alpha)}_z.
$$
Then using that $U^{(2,\alpha)}_zU^{(2,\alpha)}_z=Id_{A^2_\alpha}$ and straightforward computations, we obtain
$$
\quad \left(U^{(q,\alpha)}_z\right)^{*}U^{(p,\alpha)}_z=T_{b_z}
\quad\textnormal{ and } \quad U^{(p,\alpha)}_z\left(U^{(q,\alpha)}_z\right)^{*}=T_{b_z}^{-1},
$$
where
\begin{equation}
\label{bofz}
b_z(w)=\frac{(1-\overline{w}z)^{(n+1+\alpha)\left(\frac{1}{q}-\frac{1}{p}\right)}}{(1-\overline{z}w)^{(n+1+\alpha)\left(\frac{1}{q}-\frac{1}{p}\right)}}.
\end{equation}
Also observe at this point that if $p=q=2$ then $b_z(w)=1$.  This will be important later on when we consider the special case of $A^2_\alpha$.

For $z\in\B_n$ and $S\in\mathcal{L}(A^p_{\alpha}, A^p_{\alpha})$ we then define the map
$$
S_z :=U^{(p,\alpha)}_z S(U^{(q,\alpha)}_z)^{*},
$$
which induces a map $\Psi_S:\B_n\to \mathcal{L}(A^p_{\alpha}, A^p_{\alpha})$ given by $$\Psi_S(z)=S_z.$$  One should think of the map $S_z$ in the following way.
This is an operator on $A^p_\alpha$ and so it first acts as ``translation'' in $\B_n$, then the action of $S$, then ``translation'' back.  We now show how to extend the map $\Psi_S$ continuously to a map from $M_\mathcal{A}$ to $\mathcal{L}(A^p_{\alpha}, A^p_{\alpha})$ when endowed with both the weak and strong operator topologies.

First, observe that $C(\overline{\B_n})\subset\mathcal{A}$ induces a natural projection $\pi:M_{\mathcal{A}}\to M_{C(\overline{\B_n})}$.  If $x\in M_{\mathcal{A}}$, let
\begin{equation}
\label{bofx}
b_x(w)=\frac{(1-\overline{w}\pi(x))^{(n+1+\alpha)\left(\frac{1}{q}-\frac{1}{p}\right)}}{(1-\overline{\pi(x)}w)^{(n+1+\alpha)\left(\frac{1}{q}-\frac{1}{p}\right)}}.
\end{equation}
So, when $z_\omega$ is a net in $\B_n$ that tends to $x\in M_{\mathcal{A}}$,
then $z_{\omega}=\pi(z_{\omega})\to \pi(x)$ in the Euclidean metric, and so we have $b_{z_{\omega}}\to b_x$
uniformly on compact sets of $\B_n$ and boundedly.  Furthermore,
$$
(U^{(q,\alpha)}_z)^{*}U^{(p,\alpha)}_z=T_{b_z}\to T_{b_x}
\ \textnormal{ and }\ (U^{(p,\alpha)}_z)^{*}  U^{(q,\alpha)}_z=T_{\overline{b_z}}\to T_{\overline{b_x}},
$$
where convergence is in the strong operator topologies of $\mathcal{L}(A^p_{\alpha},A^p_{\alpha})$ and
$\mathcal{L}(A^q_{\alpha},A^q_{\alpha})$, respectively.  If $a\in\mathcal{A}$ then Lemma \ref{Maximal3}
implies $a\circ \varphi_{z_{\omega}}\to a\circ\varphi_x$ uniformly on compact sets of $\B_n$.
The above discussion implies that
$$
T_{(a\circ\varphi_{z_{\omega}} ) b_{z_{\omega}}}\to T_{(a\circ\varphi_x)b_x}
$$
in the strong operator topology associated with $\mathcal{L}(A^p_{\alpha},A^p_{\alpha})$.

Recall that we have $k_{z}^{(p,\alpha)}(w)=\frac{(1-\abs{z}^2)^{\frac{n+1+\alpha}{q}}}{(1-\overline{z}w)^{n+1+\alpha}}$,
with $\norm{k^{(p,\alpha)}_z}_{A^p_{\alpha}}\approx 1$, and so
$$
(1-\abs{\xi}^2)^{\frac{n+1+\alpha}{p}}J_z^{\frac{2}{p}}(\xi)=(1-\abs{\varphi_z(\xi)}^2)^{\frac{n+1+\alpha}{p}}\frac{\abs{1-\overline{z}\xi}^{\frac{2}{p}(n+1+\alpha)}}{(1-\xi\overline{z})^{\frac{2}{p}(n+1+\alpha)}}=(1-\abs{\varphi_z(\xi)}^2)^{\frac{n+1+\alpha}{p}}\lambda_{(p,\alpha)}(\xi,z).
$$
Here the constant $\lambda_{(p,\alpha)}$ is unimodular, and will essentially be the eigenvalue of the operator $\left(U^{(p,\alpha)}_z\right)^*$.  To see this, if $f\in A^{p}_{\alpha}$, then
\begin{eqnarray}
\ip{f}{\left(U^{(p,\alpha)}_z\right)^* k_\xi^{(q,\alpha)}}_{A^2_{\alpha}} & = &
\ip{U^{(p,\alpha)}_zf}{ k_\xi^{(q,\alpha)}}_{A^2_{\alpha}}=
\ip{J_z^{\frac{2}{p}} (f\circ\varphi_z)}{ k_\xi^{(q,\alpha)}}_{A^2_{\alpha}}\notag\\
& = & f(\varphi_z(\xi))(1-\abs{\xi}^2)^{\frac{n+1+\alpha}{p}}J_z^{\frac{2}{p}}(\xi)\notag\\
& = & f(\varphi_z(\xi))(1-\abs{\varphi_z(\xi)}^2)^{\frac{n+1+\alpha}{p}}\lambda_{(p,\alpha)}(\xi,z)\notag\\
& = & \ip{f}{\overline{\lambda_{(p,\alpha)}(\xi,z)} k_{\varphi_z(\xi)}^{(q,\alpha)}}_{A^2_{\alpha}}\notag.
\end{eqnarray}
This computation yields
\begin{equation}
\label{Comp}
\left(U^{(p,\alpha)}_z\right)^* k_\xi^{(q,\alpha)}=\lambda_{(p,\alpha)}(\xi,z)k_{\varphi_z(\xi)}^{(q,\alpha)}.
\end{equation}
We use these computations to study the continuity of the above map as a function of $z$.
\begin{lm}
\label{UniformCon}
Fix $\,\xi\in\B_n$.  Then the map $z\mapsto \left(U^{(p,\alpha)}_z\right)^* k_\xi^{(q,\alpha)}$ is uniformly continuous from $(\B_n,\rho)$ into $(A^{q}_{\alpha},\norm{\,\cdot\,}_{A^q_{\alpha}})$.
\end{lm}
\begin{proof}
By \eqref{Comp} we only need to prove the maps $z\mapsto \lambda_{(p,\alpha)}(z,\xi)$ and
$z\mapsto k_{\varphi_z(\xi)}^{(q,\alpha)}$ are uniformly continuous from $(\B_n,\rho)$ into
$(\C,\abs{\,\cdot\,})$ and $(A^{q}_{\alpha},\norm{\,\cdot\,}_{A^q_{\alpha}})$, respectively.
It is obvious $z\mapsto \lambda_{(p,\alpha)}(z,\xi)$  has the desired property.
So, we focus only in the continuity of the second map.

By Lemma \ref{Invariance}, we have that $z\mapsto\varphi_z(\xi)$ is uniformly continuous from $(\B_n,\rho)$ into itself.  So, it suffices to prove the uniform continuity of the map $w\mapsto k^{(q,\alpha)}_w$.  Namely, for any $\epsilon>0$, there is a $\delta>0$ such that if $\abs{w}<\delta$ then
$$
\sup_{z\in\B_n}\norm{k_{z}^{(q,\alpha)}-k_{\varphi_z(w)}^{(q,\alpha)}}_{A^q_{\alpha}}<\epsilon.
$$
We use the duality between $A^{p}_{\alpha}$ and $A^{q}_{\alpha}$ to have that
$$
\sup_{z\in\B_n}\norm{k_{z}^{(q,\alpha)}-k_{\varphi_z(w)}^{(q,\alpha)}}_{A^q_{\alpha}}\approx\sup_{z\in\B_n}\sup_{f\in A^{p}_{\alpha}}\abs{(1-\abs{z}^2)^{\frac{n+1+\alpha}{p}}f(z)-(1-\abs{\varphi_z(w)}^2)^{\frac{n+1+\alpha}{p}}f(\varphi_z(w))}.
$$
Consider the term inside the supremums, and observe it can be dominated by
$$
\abs{(1-\abs{z}^2)^{\frac{n+1+\alpha}{p}}\left(f(z)-f(\varphi_z(w))\right)} +\abs{f(\varphi_z(w))}\abs{(1-\abs{z}^2)^{\frac{n+1+\alpha}{p}}-(1-\abs{\varphi_z(w)}^2)^{\frac{n+1+\alpha}{p}}}
$$
by adding and subtracting a common term.  For the second term, it is easy to see using the reproducing property of the kernel $k_z^{(p,\alpha)}$ that this is dominated by
$$
C_{p,\alpha}\norm{f}_{A^p_{\alpha}}
\abs{1-\frac{\abs{1-\overline{w} z}^{2\frac{n+1+\alpha}{p}}}{(1-\abs{w}^2)^{\frac{n+1+\alpha}{p}}}},
$$
and the last expression can be made as small we wish, independently of $z$, by taking $\abs{w}$ small.  For the first term, observe that
$$
(1-\abs{z}^2)^{\frac{n+1+\alpha}{p}}\abs{(f(z)-f(\varphi_z(w)))} \lesssim \norm{f}_{A^p_{\alpha}}\norm{K_z-K_{\varphi_z}}_{L^\infty}.
$$
Again, this last estimate can be made as small as desired.
\end{proof}

\begin{prop}
\label{WOTCon}
Let $S\in \mathcal{L}(A^p_{\alpha}, A^p_{\alpha})$.  Then the map $\Psi_S: \B_n \to \left(\mathcal{L}(A^p_{\alpha}, A^p_{\alpha}), WOT\right)$ extends continuously to $M_{\mathcal{A}}$.
\end{prop}
\begin{proof}
Bounded sets in $\mathcal{L}(A^p_{\alpha}, A^p_{\alpha})$ are metrizable and have compact closure in the weak operator topology.  Since $\Psi_S(\B_n)$ is bounded, by Lemma \ref{Extend}, we only need to show $\Psi_S$ is uniformly continuous from $(\B_n,\rho)$ into $\left(\mathcal{L}(A^p_{\alpha}, A^p_{\alpha}), WOT\right)$, where $WOT$ is the weak operator topology.  Namely, we need to demonstrate that for $f\in A^p_{\alpha}$ and $g\in A^{q}_{\alpha}$ the function $z\mapsto \ip{S_z f}{g}_{A^2_{\alpha}}$ is uniformly continuous from $(\B_n,\rho)$ into $(\C, \abs{\,\cdot\,})$.

For $z_1,z_2\in\B_n$ we have
\begin{eqnarray*}
S_{z_1}-S_{z_2} & = & U_{z_1}^{(p,\alpha)}S(U_{z_1}^{(q,\alpha)})^*-U_{z_2}^{(p,\alpha)}S(U_{z_2}^{(q,\alpha)})^*\\
 & = & U_{z_1}^{(p,\alpha)}S[(U_{z_1}^{(q,\alpha)})^*-(U_{z_2}^{(q,\alpha)})^*]+(U_{z_1}^{(p,\alpha)}-U_{z_2}^{(p,\alpha)})S(U_{z_2}^{(q,\alpha)})^*\\
 & = & A+B.
\end{eqnarray*}
The terms $A$ and $B$ have a certain symmetry, and so it is enough to deal with either,
since the argument will work in the other case as well.  Observe that
\begin{eqnarray*}
\abs{\ip{Af}{g}_{A^2_{\alpha}}} & \leq & \norm{U_{z_1}^{(p,\alpha)}S}_{\mathcal{L}(A^p_{\alpha}, A^p_{\alpha})}\norm{[(U_{z_1}^{(q,\alpha)})^*-(U_{z_2}^{(q,\alpha)})^*]f}_{A^p_{\alpha}}\norm{g}_{A^q_{\alpha}}\\
\abs{\ip{Bf}{g}_{A^2_{\alpha}}} & \leq & \norm{(U_{z_1}^{(q,\alpha)})^{*}S}_{\mathcal{L}(A^p_{\alpha}, A^p_{\alpha})}\norm{[(U_{z_1}^{(p,\alpha)})^*-(U_{z_2}^{(p,\alpha)})^*]g}_{A^q_{\alpha}}\norm{f}_{A^p_{\alpha}}.
\end{eqnarray*}
Since $S$ is bounded and since $\norm{U_{z}^{(p,\alpha)}}_{\mathcal{L}(A^p_{\alpha}, A^p_{\alpha})}\leq C(p,\alpha)$ for all $z$, we just need to show the expression
$$
\norm{[(U_{z_1}^{(p,\alpha)})^*-(U_{z_2}^{(p,\alpha)})^*]g}_{A^q_{\alpha}}
$$
can be made small.  It suffices to do this on a dense set of functions,
and in particular we can take the linear span of $\left\{k_{w}^{(p,\alpha)}:w\in\B_n\right\}$.
Then we can apply Lemma \ref{UniformCon} to conclude the result.
\end{proof}

This Proposition allows us to define $S_x$ for all $x\in M_{\mathcal{A}}$. Namely, we let $S_x:=\Psi_S(x)$. In particular, if $(z_{\omega})$ is a net in $\B_n$ tending to $x\in M_{\mathcal{A}}$ then $S_{z_\omega}\to S_x$ in the weak operator topology. In Proposition~\ref{SOTCon} below we will show that if $S\in\mathcal{T}_{p,\alpha}$ then we also have $S_{z_\omega}\to S_x$ in the strong operator topology.

\begin{lm}
\label{Inverse}
If $(z_\omega)$ is a net in $\B_n$ converging to $x\in M_{\mathcal{A}}$ then $T_{b_x}$ is invertible and $T_{b_{z_{\omega}}}^{-1}\to T_{b_x}^{-1}$ in the strong operator topology.
\end{lm}

\begin{proof}
By Proposition \ref{WOTCon} applied to the operator $S=Id_{A^p_\alpha}$ we have that
$U_{z_{\omega}}^{(p,\alpha)}\left(U_{z_{\omega}}^{(q,\alpha)}\right)^*=T_{b_{z_{\omega}}}^{-1}$ has a weak operator limit
in $\mathcal{L}(A^p_{\alpha}, A^p_{\alpha})$, denote this by $Q$.  The Uniform Boundedness Principle then says that
there is a constant $C$ such that $\norm{T_{b_{z_{\omega}}}^{-1}}_{\mathcal{L}(A^p_{\alpha}, A^p_{\alpha})}\leq C$
for all $\omega$.  Then, given $f\in A^p_{\alpha}$ and $g\in A^q_{\alpha}$, since we know that
$$
\norm{\left(T_{\overline{b}_{z_{\omega}}}-T_{\overline{b}_{x}}\right)g}_{A^q_{\alpha}}\to 0,
$$
we have
\begin{eqnarray*}
\ip{T_{b_x}Qf}{g}_{A^2_{\alpha}}=\ip{Qf}{T_{\overline{b}_x}g}_{A^2_{\alpha}} & = & \lim_{\omega}\ip{T^{-1}_{b_{z_{\omega}}}f}{T_{\overline{b}_x}g}_{A^2_{\alpha}}\\
& = & \lim_{\omega}\left(\ip{T^{-1}_{b_{z_{\omega}}}f}{\left(T_{\overline{b}_x}-T_{\overline{b}_{z_{\omega}}}\right)g}_{A^2_{\alpha}}+\ip{T^{-1}_{b_{z_{\omega}}}f}{T_{\overline{b}_{z_{\omega}}}g}_{A^2_{\alpha}}\right)\\
& = & \ip{f}{g}_{A^2_{\alpha}}+\lim_{\omega}\ip{T^{-1}_{b_{z_{\omega}}}f}{\left(T_{\overline{b}_x}-T_{\overline{b}_{z_{\omega}}}\right)g}_{A^2_{\alpha}}=\ip{f}{g}_{A^2_\alpha}.
\end{eqnarray*}
This gives $T_{b_x}Q=Id_{A^p_{\alpha}}$.  Since taking adjoints is a continuous operation in the $WOT$,
$T_{\overline{b}_{z_{\omega}}}^{-1}\to Q^*$, and interchanging the roles of $p$ and $q$,
we have $T_{\overline{b}_x}Q^*=Id_{A^q_{\alpha}}$, which implies that $QT_{b_x}=Id_{A^p_{\alpha}}$.
So, $Q=T_{b_x}^{-1}$ and $T_{b_{z_{\omega}}}^{-1}\to T_{b_x}^{-1}$ in the weak operator topology.  Finally,
$$
T_{b_{z_{\omega}}}^{-1}-T_{b_x}^{-1}=T_{b_{z_{\omega}}}^{-1}\left(T_{b_x}-T_{b_{z_{\omega}}}\right) T_{b_x}^{-1},
$$
and since $\norm{T_{b_{z_{\omega}}}^{-1}}_{\mathcal{L}(A^p_{\alpha}, A^p_{\alpha})}\leq C$ and $T_{b_{z_{\omega}}}-T_{b_x}\to 0$
in the strong operator topology, we also have $T_{b_{z_{\omega}}}^{-1}\to T_{b_x}^{-1}$ in the strong operator topology as claimed.
\end{proof}

\begin{prop}
\label{SOTCon}
If $S\in\mathcal{T}_{p,\alpha}$ and $(z_{\omega})$ is a net in $\B_n$ that tends to $x\in M_{\mathcal{A}}$, then $S_{z_{\omega}}\to S_x$ in the strong operator topology.  In particular, $\Psi_S:\B_n\to (\mathcal{L}(A^p_{\alpha},A^p_{\alpha}), SOT)$ extends continuously to $M_{\mathcal{A}}$.
\end{prop}
\begin{proof}
First observe that if $A, B\in \mathcal{L}(A^p_{\alpha}, A^p_{\alpha})$ then
\begin{eqnarray*}
(AB)_z= U^{(p,\alpha)}_zAB(U^{(q,\alpha)}_z)^{*} & = & U^{(p,\alpha)}_zA(U^{(q,\alpha)}_z)^{*}(U^{(q,\alpha)}_z)^{*} U^{(p,\alpha)}_zU^{(p,\alpha)}_zB(U^{(q,\alpha)}_z)^{*}\\
 & = & A_z T_{b_z}B_z.
\end{eqnarray*}
In general, this applies to longer products of operators.

For $S\in \mathcal{T}_{p,\alpha}$ and $\epsilon>0$, by Theorem \ref{Sua2Thm}
there is a finite sum of finite products of Toeplitz operators with symbols in $\mathcal{A}$ such that
$\norm{R-S}_{\mathcal{L}(A^p_{\alpha},A^p_{\alpha})}<\epsilon$, and so
$\norm{R_z-S_z}_{\mathcal{L}(A^p_{\alpha},A^p_{\alpha})}<C(p,\alpha)\epsilon$.
Passing to the $WOT$ limit we have $\norm{R_x-S_x}_{\mathcal{L}(A^p_{\alpha},A^p_{\alpha})}<C(p,\alpha)\epsilon$
for all $x\in M_{\mathcal{A}}$.
These observations imply that it suffices to prove the Lemma for $R$ alone, and then by linearity,
it suffices to consider the special case $R=\prod_{j=1}^m T_{a_j}$, where $a_j\in\mathcal{A}$.
A simple computation shows that
$$
U^{(2,\alpha)}_z T_a U^{(2,\alpha)}_z= T_{a\circ\varphi_z}
$$
and more generally,
\begin{eqnarray*}
(T_a)_z & = & U^{(p,\alpha)}_z\left(U^{(q,\alpha)}_z\right)^*\left(U^{(q,\alpha)}_z\right)^*T_a U^{(p,\alpha)}_zU^{(p,\alpha)}_z\left(U^{(q,\alpha)}_z\right)^*\\
& = & U^{(p,\alpha)}_z\left(U^{(q,\alpha)}_z\right)^*\left(T_{\overline{J}_z^{1-\frac{2}{q}}}U^{(2,\alpha)}_zT_aU^{(2,\alpha)}_zT_{J_z^{1-\frac{2}{p}}}\right)U^{(p,\alpha)}_z\left(U^{(q,\alpha)}_z\right)^*\\
& = & T_{b_{z}}^{-1}T_{(a\circ\varphi_z) b_z}T_{b_{z}}^{-1}.
\end{eqnarray*}
We now combine this computation with the observation at the beginning of the lemma to see that
\begin{eqnarray*}
\left(\prod_{j=1}^{m} T_{a_j}\right)_z & = & (T_{a_1})_z T_{b_z}\cdots T_{b_z}(T_{a_m})_z\\
 & = & T^{-1}_{b_z} T_{(a_1\circ\varphi_z) b_z}T^{-1}_{b_z}T_{(a_2\circ\varphi_z) b_z}T^{-1}_{b_z}\cdots
 T^{-1}_{b_z}T_{(a_m\circ\varphi_z) b_z}T^{-1}_{b_z}.
\end{eqnarray*}
But, since the product of $SOT$ nets is $SOT$ convergent, Lemma \ref{Inverse} and the fact that
$T_{(a\circ\varphi_{z_\omega})b_{z_{\omega}}}\to T_{(a\circ\varphi_x) b_x}$ in the $SOT$, give
$$
\left(\prod_{j=1}^{m} T_{a_j}\right)_{z_\alpha}\to T^{-1}_{b_x}
T_{(a_1\circ\varphi_x) b_x}T^{-1}_{b_x}T_{(a_2\circ\varphi_x) b_x}T^{-1}_{b_x}\cdots T^{-1}_{b_x}
T_{(a_m\circ\varphi_x) b_x}T^{-1}_{b_x}.
$$
But this is exactly the statement $R_{z_{\omega}}\to R_x$ in the $SOT$ for the operator $\prod_{j=1}^{m} T_{a_j}$,
and proves the claimed continuous extension.
\end{proof}

The next result gives information about the Berezin transform vanishing in terms of the operators $S_x$.

\begin{prop}
\label{BerezinVanish}
Let $S\in\mathcal{L}(A^p_{\alpha},A^p_{\alpha})$.  Then $B(S)(z)\to 0$ as $\abs{z}\to 1$ if and only if\/ $S_x=0$
for all\/ $x\in M_{\mathcal{A}}\setminus\B_n$.
\end{prop}
\begin{proof}
If $z,\xi\in\B_n$, then we have
\begin{eqnarray*}
B(S_z)(\xi) & = & \ip{S\left(U^{(q,\alpha)}\right)^*k_\xi^{(p,\alpha)}}{\left(U^{(p,\alpha)}\right)^*k_\xi^{(q,\alpha)}}_{A^2_{\alpha}}\\
& = & \lambda_{(q,\alpha)}(\xi,z)\overline{\lambda_{(p,\alpha)}(\xi,z)}\ip{Sk_{\varphi_z(\xi)}^{(p,\alpha)}}{k_{\varphi_z(\xi)}^{(q,\alpha)}}_{A^2_{\alpha}}.
\end{eqnarray*}
Thus, $\abs{B(S_z)(\xi)}=\abs{B(S)(\varphi_z(\xi))}$ since $\lambda_{(p,\alpha)}$ and $\lambda_{(q,\alpha)}$
are unimodular numbers.
For $x\in M_{\mathcal{A}}\setminus\B_n$ and $\xi\in\B_n$ fixed, if $(z_\omega)$ is a net in $\B_n$
tending to $x$, the continuity of $\Psi_S$ in the $WOT$ and Proposition \ref{WOTCon} give that
$B(S_{z_\omega})(\xi)\to B(S_x)(\xi)$, and consequently $\abs{B(S)(\varphi_{z_\omega}(\xi))}\to \abs{B(S_x)(\xi)}$.

Now, suppose that $B(S)(z)$ vanishes as $\abs{z}\to 1$.  Since $x\in M_{\mathcal{A}}\setminus\B_n$ and $z_\omega\to x$,
we have that $\abs{z_\omega}\to 1$, and similarly $\abs{\varphi_{z_{\omega}}(\xi)}\to 1$.
Since $B(S)(z)$ vanishes as we approach the boundary, $B(S_x)(\xi)=0$, and since $\xi\in\B_n$ was arbitrary and the
Berezin transform is one-to-one, we see that $S_x=0$.

Conversely, suppose that the Berezin transform does not vanish as we approach the boundary.
Then there is a sequence $\{z_k\}$ in $\B_n$ such that $\abs{z_k}\to 1$ and $\abs{B(S)(z_k)}\geq\delta>0$.
Since $M_{\mathcal{A}}$ is compact, we can extract a subnet $(z_{\omega})$ of $\{z_k\}$ converging in $M_{\mathcal{A}}$ to
$x\in M_{\mathcal{A}}\setminus\B_n$.  The computations above imply $\abs{B(S_x)(0)}\geq\delta>0$, which gives that $S_x\neq 0$.
\end{proof}

\section{Characterization of the Essential Norm on \texorpdfstring{$A^p_{\alpha}\ $}{Weighted Bergman Spaces}}
\label{Characterization}

We have now collected enough tools to provide a characterization of the essential norm of an operator on $A^p_{\alpha}$.
Fix $\varrho>0$ and let $\{w_m\}$ and $D_m$ be the sets of Lemma \ref{StandardGeo}.  Define the measure
$$
\mu_{\varrho}:=\sum_{m=1}^\infty v_{\alpha}(D_m) \delta_{w_m},
$$
which is well-known to be a $A^p_{\alpha}$ Carleson measure, and so $T_{\mu_\varrho}:A^p_{\alpha}\to A^p_{\alpha}$ is bounded.  The following Lemma can be deduced from results in Luecking, \cite{L}, or work of Coifman and Rochberg, \cite{CR}, and we omit the proof.
\begin{lm}
$T_{\mu_\varrho}\to Id_{A^p_\alpha}$ on $\mathcal{L}(A^p_{\alpha},A^p_{\alpha})$ when $\varrho\to 0$.
\end{lm}

Now choose $0<\varrho\leq 1$ so that $\norm{T_{\mu_\varrho}-Id_{A^p_\alpha}}_{\mathcal{L}(A^p_{\alpha},A^p_{\alpha})}<\frac{1}{4}$ and consequently $\norm{T_{\mu_\varrho}}_{\mathcal{L}(A^p_{\alpha}, A^p_{\alpha})}$ and $\norm{T_{\mu_\varrho}^{-1}}_{\mathcal{L}(A^p_{\alpha}, A^p_{\alpha})}$ are less than $\frac{3}{2}$.  Fix this value of $\varrho$, and denote $\mu_\varrho:=\mu$ for the rest of the paper.

For $S\in\mathcal{L}(A^p_{\alpha}, A^p_{\alpha})$ and $r>0$, let
$$
\mathfrak{a}_S(r):=\varlimsup_{\abs{z}\to 1}\sup\left\{\norm{Sf}_{A^p_{\alpha}}: f\in T_{\mu 1_{D(z,r)}}(A^p_{\alpha}),\norm{f}_{A^p_{\alpha}}\leq 1\right\},
$$
and then define
$$
\mathfrak{a}_S:=\lim_{r\to 1}\mathfrak{a}_S(r).
$$
Since for $r_1<r_2$ we have $T_{\mu 1_{D(z,r_1)}}(A^p_{\alpha})\subset T_{\mu 1_{D(z,r_2)}}(A^p_{\alpha})$ and
$\mathfrak{a}_S(r)\leq\norm{S}_{\mathcal{L}(A^p_{\alpha},A^p_{\alpha})}$, this limit is well defined.
We define two other measures of the size of an operator which are given in a very intrinsic and geometric way:
\begin{eqnarray*}
\mathfrak{b}_{S} & := & \sup_{r>0} \varlimsup_{\abs{z}\to 1}\norm{M_{1_{D(z,r)}}S} _{\mathcal{L}(A^p_{\alpha},L^p_{\alpha})},\\
\mathfrak{c}_{S} & := & \lim_{r\to 1}\norm{M_{1_{(r\B_n)^c}}S}_{\mathcal{L}(A^p_{\alpha},L^p_{\alpha})}.
\end{eqnarray*}
In the last definition, for notational simplicity, we let $\left(r\B_n\right)^c=\B_n\setminus r\B_n$.  Finally, for $S\in \mathcal{L}(A^p_{\alpha},A^p_{\alpha})$ recall that
$$
\norm{S}_e=\inf\left\{\norm{S-Q}_{\mathcal{L}(A^p_{\alpha},A^p_{\alpha})}:Q\textnormal{ is compact}\right\}.
$$

We first show how to compute the essential norm of an operator $S$ in terms of the operators $S_x$,
where $x\in M_{\mathcal{A}}\setminus\B_n$.

\begin{thm}
\label{EssentialviaSx}
Let $S\in\mathcal{T}_{p,\alpha}$.  Then there exists a constant $C(p,\alpha,n)$ such that
\begin{equation}
\label{EssentialNorm&Sx}
\sup_{x\in M_{\mathcal{A}}\setminus\B_n}\norm{S_x}_{\mathcal{L}(A^p_\alpha, A^p_{\alpha})}\lesssim\norm{S}_e\lesssim \sup_{x\in M_{\mathcal{A}}\setminus\B_n}\norm{S_x}_{\mathcal{L}(A^p_\alpha, A^p_{\alpha})}.
\end{equation}
\end{thm}
\begin{proof}
For $S$ compact, \eqref{EssentialNorm&Sx} is easy to demonstrate.
Since $k_\xi^{(p,\alpha)}\to 0$ weakly as $\abs{\xi}\to 1$, then $\norm{Sk_\xi^{(p,\alpha)}}_{A^p_{\alpha}}$ goes to $0$ as well.
Thus, we have
\begin{equation}
\label{Vanishing}
\abs{B(S)(\xi)}=\abs{\ip{S k_\xi^{(p,\alpha)}}{k_\xi^{(q,\alpha)}}_{A^2_\alpha}}\leq
\norm{Sk_\xi^{(p,\alpha)}}_{A^p_{\alpha}}\norm{k_\xi^{(q,\alpha)}}_{A^q_{\alpha}}\approx\norm{Sk_\xi^{(p,\alpha)}}_{A^p_{\alpha}}.
\end{equation}
Hence, the compactness of $S$ implies that the Berezin transform vanishes as $\abs{\xi}\to 1$.
Then  Proposition \ref{BerezinVanish} gives that $S_x=0$ for all $x\in M_{\mathcal{A}}\setminus\B_n$.

Now let $S$ be any bounded operator on $A^p_{\alpha}$ and suppose that $Q$ is a compact operator on $A^p_{\alpha}$.
Select $x\in M_{\mathcal{A}}\setminus \B_n$ and a net $(z_\omega)$ in $\B_n$ tending to $x$.
Since the maps $U_{z_\omega}^{(p,\alpha)}$ and $U_{z_\omega}^{(q,\alpha)}$ are isometries on $A^p_{\alpha}$ and $A^q_{\alpha}$,
we have
$$
\norm{S_{z_{\omega}}+Q_{z_{\omega}}}_{\mathcal{L}(A^p_\alpha, A^p_{\alpha})}\leq\norm{S+Q}_{\mathcal{L}(A^p_\alpha, A^p_{\alpha})}.
$$
Since $S_{z_{\omega}}+Q_{z_{\omega}}\to S_x$ in the $WOT$, passing to the limit we get
$$
\norm{S_x}_{\mathcal{L}(A^p_\alpha, A^p_{\alpha})}\lesssim \varliminf\norm{S_{z_{\omega}}+Q_{z_{\omega}}}_{\mathcal{L}(A^p_\alpha, A^p_{\alpha})}\leq \norm{S+Q}_{\mathcal{L}(A^p_\alpha, A^p_{\alpha})},
$$
which gives
$$
\sup_{x\in M_{\mathcal{A}}\setminus\B_n}\norm{S_x}_{\mathcal{L}(A^p_\alpha, A^p_{\alpha})}\lesssim\norm{S}_e ,
$$
the first inequality in \eqref{EssentialNorm&Sx}.  It only remains to address the last inequality.  To accomplish this, we will instead prove that
\begin{equation}
\label{AlphaControlled}
\mathfrak{a}_S\lesssim \sup_{x\in M_{\mathcal{A}}\setminus\B_n}\norm{S_x}_{\mathcal{L}(A^p_\alpha, A^p_{\alpha})}.
\end{equation}
Then we compare this with the first inequality in \eqref{Redux2}, $\norm{S}_e\lesssim\mathfrak{a}_S$, shown below, to obtain
$$
\norm{S}_e\lesssim \sup_{x\in M_{\mathcal{A}}\setminus\B_n}\norm{S_x}_{\mathcal{L}(A^p_\alpha, A^p_{\alpha})}.
$$
Also note that if \eqref{AlphaControlled} holds, then
\begin{equation}
\label{LastStep}
\mathfrak{a}_S\lesssim\norm{S}_e
\end{equation}
is also true.  We now turn to addressing \eqref{AlphaControlled}.  It suffices to demonstrate that
$$
\mathfrak{a}_S(r)\lesssim\sup_{x\in M_{\mathcal{A}}\setminus\B_n}\norm{S_x}_{\mathcal{L}(A^p_\alpha,
A^p_{\alpha})}\quad\ \forall r>0.
$$
Fix a radius $r>0$. By the definition of $\mathfrak{a}_S(r)$ there is a sequence $\{z_j\}\subset\B_n$ tending to
$\partial\B_n$ and a normalized sequence of functions $f_j\in T_{\mu 1_{D(z_j,r)}}(A^p_{\alpha})$ with
$\norm{Sf_j}_{A^p_{\alpha}}\to\mathfrak{a}_S(r)$.  To each $f_j$ we have a corresponding $h_j\in A^p_{\alpha}$, and then
\begin{eqnarray*}
f_j(w)=T_{\mu 1_{D(z_j,r)}}h_j(w) & = & \sum_{w_m\in D(z_j,r)}\frac{v_\alpha(D_m)}{(1-\overline{w_m}w)^{n+1+\alpha}}h_j(w_m)\\
 & = & \sum_{w_m\in D(z_j,r)} a_{j,m} \frac{(1-\abs{w_m}^2)^{\frac{n+1+\alpha}{q}}}{(1-\overline{w_m}w)^{n+1+\alpha}}\\
 & = & \sum_{w_m\in D(z_j,r)} a_{j,m} k_{w_m}^{(p,\alpha)}(w),
\end{eqnarray*}
where $a_{j,m}=v_\alpha(D_m)(1-\abs{w_m}^2)^{-\frac{n+1+\alpha}{q}}h_j(w_m)$.  We then have that
$$
\left(U^{(q,\alpha)}_{z_j}\right)^{*}f_j(w)= \sum_{\varphi_{z_j}(w_m)\in D(0,r)} a_{j,m}' k_{\varphi_{z_j}(w_m)}^{(p,\alpha)}(w),
$$
where $a_{j,m}'$ is simply the original constant $a_{j,m}$ multiplied by the unimodular constant $\lambda_{(q,\alpha)}$.

Observe that the points $\abs{\varphi_{z_j}(w_m)}\leq\tanh r$.  For $j$ fixed, arrange the points $\varphi_{z_j}(w_m)$ such that $\abs{\varphi_{z_j}(w_m)}\leq\abs{\varphi_{z_j}(w_{m+1})}$ and $\textnormal{arg}\,\varphi_{z_j}(w_m)\leq\textnormal{arg}\, \varphi_{z_j}(w_{m+1})$.  Since the M\"obius map $\varphi_{z_j}$ preserves the hyperbolic distance between the points $\{w_m\}$ we have for $m\neq k$ that
$$
\beta(\varphi_{z_j}(w_m),\varphi_{z_j}(w_k))=\beta(w_m,w_k)\geq\frac{\varrho}{4}>0.
$$
Thus, there can only be at most $N_j\leq M(\varrho, r)$ points in the collection $\varphi_{z_j}(w_m)$ belonging to the disc $D(0,z_j)$.  By passing to a subsequence, we can assume that $N_j=M$ and is independent of $j$.

For the fixed $j$, and $1\leq m\leq M$, select $g_{j,k}\in H^\infty$ with
$\norm{g_{j,k}}_{H^\infty}\leq C(\tanh r,\frac{\varrho}{4})$, such that $g_{j,k}(\varphi_{z_j}(w_m))=\delta_{k,m}$,
the Kronecker delta, when $1\leq k\leq M$.  The existence of the functions is easy to deduce from a result of Berndtsson \cite{B}, see also \cite{Sua}.  We then have
\begin{eqnarray*}
\ip{\left(U_{z_j}^{(q,\alpha)}\right)^{*}f_j}{g_{j,k}}_{A^2_{\alpha}} & = & \sum_{\varphi_{z_j}(w_m)\in D(0,r)} a_{j,m}' g_{j,k}(\varphi_{z_j}(w_m))\left(1-\abs{\varphi_{z_j}(w_m)}^2\right)^{\frac{n+1+\alpha}{q}}\\
 & = & a_{j,k}'\left(1-\abs{\varphi_{z_j}(w_k)}^2\right)^{\frac{n+1+\alpha}{q}}.
\end{eqnarray*}
This expression implies that the sequence $\abs{a_{j,k}'}\leq C=C(n,p,\varrho,r,\alpha)$ independently of $j$ and $k$,
because $g_{j,k}\in H^\infty$ has norm controlled by $C(r,\varrho)$, $\left( U^{(q,\alpha)}_z\right)^*$ is a bounded operator, and $f_j$ is a normalized sequence of functions in $A^p_\alpha$.

Now $(\varphi_{z_j}(w_1),\ldots,\varphi_{z_j}(w_M), a_{1}',\ldots,a_{M}')\in\C^{M(n+1)}$ is a bounded sequence in $j$,
and passing to a subsequence if necessary, we can assume that converges to a point $(v_1,\ldots, v_M,a_{1}',\ldots a_{M}')$.
Here $\abs{v_k}\leq\tanh r$ and $\abs{a_k'}\leq C$.  This gives that
$$
\left(U^{(q,\alpha)}_{z_j}\right)^{*}f_j\to \sum_{k=1}^M a_k' k_{v_k}^{(p,\alpha)} := h
$$
in the $L^p_{\alpha}$ norm and moreover,
$$
\norm{\sum_{k=1}^M a_k' k_{v_k}^{(p,\alpha)}}_{L^p_{\alpha}}=\lim_{j}\norm{\left(U^{(q,\alpha)}_{z_j}\right)^{*}f_j}_{L^p_{\alpha}}\lesssim1.
$$
Since the operator $U_{z_j}^{(p,\alpha)}$ is isometric and $\norm{S_{z_j}}_{\mathcal{L}(A^p_\alpha, A^p_{\alpha})}$
is bounded independently of $j$,
$$
\mathfrak{a}_S(r)=\lim_{j}\norm{S f_j}_{A^p_\alpha}=\lim_{j}\norm{S_{z_j}(U_{z_j}^{(q,\alpha)})^* f_j}_{A^p_\alpha}=
\lim\norm{S_{z_j} h}_{A^p_\alpha}.
$$
Since $\abs{z_j}\to 1$, by using the compactness of $M_{\mathcal{A}}$ it is possible to extract a subnet $(z_\omega)$
which converges to some point $x\in M_{\mathcal{A}}\setminus\B_n$.  Then $S_{z_\omega}h\to S_x h$ in $A^p_{\alpha}$,  so
$$
\mathfrak{a}_S(r)=\lim_{\omega}\norm{S_{z_\omega} h}_{A^p_\alpha}=\norm{S_xh}_{A^p_\alpha}\lesssim\norm{S_x}_{\mathcal{L}(A^p_\alpha, A^p_{\alpha})}\lesssim\sup_{x\in M_{\mathcal{A}}\setminus\B_n}\norm{S_x}_{\mathcal{L}(A^p_\alpha, A^p_{\alpha})}.
$$
The above limit uses the continuity in the $SOT$ as guaranteed by Proposition \ref{SOTCon}.
\end{proof}

\begin{thm}
Let\/ $1<p<\infty$, $\,\alpha>-1$ and\/ $S\in \mathcal{T}_{p,\alpha}$.  Then there exist constants depending only on $n$, $p$, and $\alpha$ such that:
$$
\mathfrak{a}_S\approx\mathfrak{b}_S\approx\mathfrak{c}_S\approx\norm{S}_e.
$$
\end{thm}

\begin{proof}
By Theorem \ref{Approx3} there are Borel sets $F_j\subset G_j\subset\B_n$ such that
\begin{itemize}
\item[(i)] $\B_n=\cup F_j$;
\item[(ii)] $F_j\cap F_k=\emptyset$ if $j\neq k$;
\item[(iii)] each point of $\B_n$ lies in no more than $N(n)$ of the sets $G_j$;
\item[(iv)] $\textnormal{diam}_{\beta}\, G_j\leq d(p,S,\epsilon)$
\end{itemize}
and
\begin{equation}
\label{RecallEst}
\norm{ST_\mu-\sum_{j=1}^{\infty} M_{1_{F_j}}ST_{\mu 1_{G_j}}}_{\mathcal{L}(A^p_{\alpha}, L^p_{\alpha})}<\epsilon.
\end{equation}
Set
$$
S_m=\sum_{j=m}^{\infty} M_{1_{F_j}} S T_{\mu 1_{G_j}}.
$$
Next, we consider one more measure of the size of $S$,
$$
\varlimsup_{m\to\infty} \norm{\sum_{j=m}^{\infty} M_{1_{F_j}} S T_{\mu 1_{G_j}}}_{\mathcal{L}(A^p_{\alpha}, L^p_{\alpha})}=\varlimsup_{m\to\infty}\norm{S_m}_{\mathcal{L}(A^p_{\alpha}, L^p_{\alpha})}.
$$
First some observations.  Since every $z\in\B_n$ belongs to only $N(n)$ sets $G_j$, Lemma~\ref{CM-Cor} gives
$$
\sum_{j=m}^{\infty}\norm{T_{\mu 1_{G_j}}f}_{A^p_{\alpha}}^p\lesssim \sum_{j=1}^{\infty} \norm{1_{G_j}f}_{L^p(\mu)}^p \lesssim\norm{f}_{A^p_{\alpha}}^p.
$$
Also, since $T_\mu$ is bounded and invertible, we have that $\norm{S}_e\approx\norm{ST_\mu}_e$.  Finally, we will need to compute both norms in $\mathcal{L}(A^p_{\alpha}, A^p_{\alpha})$ and $\mathcal{L}(A^p_{\alpha}, L^p_{\alpha})$.  When necessary, we will denote the respective essential norms as $\norm{\,\cdot\,}_{e}$ and $\norm{\,\cdot\,}_{ex}$.  It is easy to show that
$$
\norm{R}_{ex}\leq\norm{R}_e\leq\norm{P_\alpha}_{L^p_{\alpha}\to A^p_{\alpha}}\norm{R}_{ex}.
$$
The strategy behind the proof of the theorem is to demonstrate the following string of inequalities
\begin{eqnarray}
\mathfrak{b}_S & \leq & \mathfrak{c}_S\,\,\lesssim \,\,\varlimsup_{m\to\infty}\norm{S_m}_{\mathcal{L}(A^p_{\alpha}, L^p_{\alpha})}\,\,\lesssim\,\,\mathfrak{b}_S\label{Redux1}\\
\norm{S}_{e} & \lesssim & \varlimsup_{m\to\infty}\norm{S_m}_{\mathcal{L}(A^p_{\alpha}, L^p_{\alpha})}\,\,\lesssim\,\, \mathfrak{a}_S\,\,\lesssim\,\, \norm{S}_e\label{Redux2}.
\end{eqnarray}

The implied constants in all these estimates depend only on $p,\alpha$ and the dimension.
Combining \eqref{Redux1} and \eqref{Redux2} we have the theorem.
We prove now the first two inequalities in \eqref{Redux2}.

Fix  $f\in A^p_{\alpha}$ of norm 1 and note that
\begin{eqnarray}
\norm{S_m f}^p_{L^p_{\alpha}} & = & \sum_{j=m}^{\infty} \norm{M_{1_{F_j}} ST_{\mu 1_{G_j}}f}_{L^p_{\alpha}}^p\notag\\
& = & \sum_{j=m}^{\infty}\left(\frac{\norm{M_{1_{F_j}} ST_{\mu 1_{G_j}}f}_{L^p_{\alpha}}}{\norm{T_{\mu 1_{G_j}}f}_{A^p_{\alpha}}}\right)^{p}\norm{T_{\mu 1_{G_j}}f}_{A^p_{\alpha}}^p\notag\\
& \leq & \sup_{j\geq m}\sup\left\{\norm{M_{1_{F_j}} Sg}^p_{L^p_{\alpha}}: g\in T_{\mu 1_{G_j}}(A^p_{\alpha}),\norm{g}_{A^p_{\alpha}}= 1\right\}\sum_{j\geq m}\norm{T_{\mu 1_{G_j}}f}_{A^p_{\alpha}}^p\notag\\
& \lesssim & \sup_{j\geq m}\sup\left\{\norm{M_{1_{F_j}} Sg}^p_{L^p_{\alpha}}: g\in T_{\mu 1_{G_j}}(A^p_{\alpha}),\norm{g}_{A^p_{\alpha}}= 1\right\}\label{Last}.
\end{eqnarray}
Since $\textnormal{diam}_{\beta}\, G_j\leq d$, by selecting $z_j\in G_j$ we have $G_j\subset D(z_j,d)$, and so $T_{\mu 1_{G_j}}(A^p_\alpha)\subset T_{1_{\mu D(z_j,d)}}(A^p_\alpha)$.
Since $z_j$ approaches the boundary, we can select an additional sequence $0<\gamma_m<1$ tending to $1$ such that
$\abs{z_j}\geq\gamma_m$ when $j\geq m$.  Using \eqref{Last} we find that
\begin{eqnarray}
\norm{S_m}_{\mathcal{L}(A_\alpha^p,L_\alpha^p)} & \lesssim & \sup_{j\geq m}\sup\left\{\norm{M_{1_{F_j}} Sg}_{L^p_{\alpha}}: g\in T_{\mu 1_{G_j}}(A^p_{\alpha}),\norm{g}_{A^p_{\alpha}}= 1\right\}\notag\\
 & \lesssim & \sup_{\abs{z_j}\geq \gamma_m}\sup\left\{\norm{M_{1_{D(z_j,d)}} Sg}_{L^p_{\alpha}}: g\in T_{\mu 1_{D(z_j,d)}}(A^p_{\alpha}),\norm{g}_{A^p_{\alpha}}= 1\right\}\label{important}\\
 & \lesssim & \sup_{\abs{z_j}\geq \gamma_m}\sup\left\{\norm{Sg}_{L^p_{\alpha}}: g\in T_{\mu 1_{D(z_j,d)}}(A^p_{\alpha}),\norm{g}_{A^p_{\alpha}}= 1\right\}\notag.
\end{eqnarray}
Since $\gamma_m\to 1$ as $m\to \infty$, we get
$$
\varlimsup_{m\to\infty} \norm{S_m}_{\mathcal{L}(A_\alpha^p,L_\alpha^p)}\lesssim \mathfrak{a}_S(d).
$$
From \eqref{RecallEst} we see that
$$
\norm{ST_\mu}_{ex}\leq\varlimsup_{m\to\infty} \norm{S_m}_{\mathcal{L}(A_\alpha^p,L_\alpha^p)}+\epsilon\lesssim \mathfrak{a}_S(d)+\epsilon\lesssim\mathfrak{a}_S+\epsilon,
$$
giving $\norm{ST_\mu}_{ex}\leq\varlimsup_{m\to\infty} \norm{S_m}_{\mathcal{L}(A_\alpha^p,L_\alpha^p)}\lesssim\mathfrak{a}_S$,
since $\epsilon$ is arbitrary.  Therefore,
\begin{equation}
\norm{S}_e\approx\norm{ST_\mu}_{e}\lesssim \norm{ST_\mu}_{ex}\leq \varlimsup_{m} \norm{S_m}_{\mathcal{L}(A_\alpha^p,L_\alpha^p)}\lesssim \mathfrak{a}_S.
\end{equation}
This gives the first two inequalities in \eqref{Redux2}.
The remaining inequality is simply \eqref{LastStep}, which was proved in Theorem \ref{EssentialviaSx}.

We now consider \eqref{Redux1}.  If $0<r<1$, there exists a positive integer $m(r)$ such that $\bigcup_{j<m(r)}F_j\subset r\B_n$.
Then
\begin{eqnarray*}
\norm{M_{1_{(r\B_n)^c}}S}_{\mathcal{L}(A_\alpha^p,L_\alpha^p)}\norm{T^{-1}_\mu}^{-1}_{\mathcal{L}(A_\alpha^p,A_\alpha^p)} & \leq & \norm{M_{1_{(r\B_n)^c}}ST_\mu}_{\mathcal{L}(A_\alpha^p,L_\alpha^p)}\\
& \leq & \norm{M_{1_{(r\B_n)^c}}\left(ST_\mu-\sum_{j=1}^{\infty} M_{1_{F_j}}ST_{1_{G_j}\mu}\right)}_{\mathcal{L}(A_\alpha^p,L_\alpha^p)}\\
& & +\norm{M_{1_{(r\B_n)^c}}\sum_{j=1}^{\infty} M_{1_{F_j}}ST_{1_{G_j}\mu}}_{\mathcal{L}(A_\alpha^p,L_\alpha^p)}\\
& \leq & \epsilon +\norm{\sum_{j=m(r)}^{\infty} M_{1_{F_j}}ST_{1_{G_j}\mu}}_{\mathcal{L}(A_\alpha^p,L_\alpha^p)}=\epsilon+\norm{S_{m(r)}}_{\mathcal{L}(A_\alpha^p,L_\alpha^p)}.
\end{eqnarray*}
This string of inequalities easily yields
\begin{equation}
\label{Redux1-1}
\mathfrak{c}_{S}=\varlimsup_{r\to 1}\norm{M_{1_{(r\B_n)^c}}S}_{\mathcal{L}(A^p_{\alpha},L^p_{\alpha})}\lesssim \varlimsup_{m\to\infty} \norm{S_m}_{\mathcal{L}(A_\alpha^p,L_\alpha^p)}.
\end{equation}
Also, \eqref{important} gives that
\begin{equation}
\label{Redux1-2}
\varlimsup_{m\to\infty} \norm{S_m}_{\mathcal{L}(A_\alpha^p,L_\alpha^p)}\lesssim\varlimsup_{\abs{z}\to 1}
\norm{M_{1_{D(z,r)}}S} _{\mathcal{L}(A^p_{\alpha},L^p_{\alpha})}\lesssim\mathfrak{b}_S.
\end{equation}
Combining the trivial inequality $\mathfrak{b}_S\leq\mathfrak{c}_S$ with \eqref{Redux1-1} and \eqref{Redux1-2} we obtain \eqref{Redux1}.
\end{proof}

From these Theorems we can deduce two results of interest.
\begin{cor}
Let $\alpha>-1$ and $1<p<\infty$ and $S\in\mathcal{T}_{p,\alpha}$.  Then
$$
\norm{S}_e\approx \sup_{\norm{f}_{A^p_\alpha}=1}\varlimsup_{\abs{z}\to 1}\norm{S_z f}_{A^p_\alpha}.
$$
\end{cor}
\begin{proof}
It is easy to see from Lemma \ref{SOTCon} and the compactness of $M_{\mathcal{A}}$ that
$$
\sup_{x\in M_{\mathcal{A}}\setminus\B_n}\norm{S_xf}_{A^p_{\alpha}}=\varlimsup_{\abs{z}\to 1}\norm{ S_z f}_{A^p_{\alpha}}.
$$
But then,
$$
\sup_{x\in M_{\mathcal{A}}\setminus\B_n}\norm{S_x}_{\mathcal{L}(A^p_{\alpha},A^p_{\alpha})}=\sup_{\norm{f}_{A^p_\alpha}=1}\varlimsup_{\abs{z}\to 1}\norm{S_z f}_{A^p_\alpha}.
$$
The result then follows from Theorem \ref{EssentialviaSx}.
\end{proof}

The next result gives the characterization of compact operators in terms of the Berezin transform and membership
in the Toeplitz algebra.

\begin{thm}
Let\/ $1<p<\infty$, $\alpha>-1$ and\/ $S\in\mathcal{L}(A^p_{\alpha},A^p_{\alpha})$.  Then $S$ is compact if and only if\/
$S\in\mathcal{T}_{p,\alpha}$ and $B(S)=0$ on $\partial\B_n$.
\end{thm}
\begin{proof}
If $B(S)=0$ on $\partial\B_n$, Proposition \ref{BerezinVanish} says that $S_x=0$ for all $x\in M_{\mathcal{A}}\setminus\B_n$.
So, if $S\in \mathcal{T}_{p,\alpha}$, Theorem \ref{EssentialviaSx} gives that $S$ must be compact.

In the other direction, if $S$ is compact then $B(S)=0$ on $\partial\B_n$ by \eqref{Vanishing}.
So it only remains to show that $S\in\mathcal{T}_{p,\alpha}$.
Since every compact operator on $A^p_{\alpha}$ can be approximated by finite rank operators,
it suffices to show that all rank one operators are in $\mathcal{T}_{p,\alpha}$.
But, the rank one operators have the form $f\otimes g$, given by
$$
(f\otimes g) (h)=\ip{h}{g}_{A^2_\alpha}f,
$$
where $f\in A^p_{\alpha}$, $g\in A^q_{\alpha}$, and $h\in A^p_{\alpha}$.
We can further suppose that $f$ and $g$ are polynomials, since the polynomials are dense in $A^p_\alpha$ and $A^q_\alpha$,
respectively.  But then
$$
f\otimes g= T_f(1\otimes 1)T_{\overline{g}},
$$
and it suffices to show that $1\otimes 1\in \mathcal{T}_{p,\alpha}$.  This is an immediate consequence of Theorem \ref{new2},
since
$1\otimes 1=T_{\delta_0}$, where $\delta_0$ is the Dirac measure concentrated at zero.
\end{proof}

\subsection{The Hilbert Space Case}

When $p=2$, some of the previous results can be strengthened in straightforward ways.
It is easy to see that if $T\in\mathcal{L}(A^2_\alpha,A^2_\alpha)$, $S\in\mathcal{T}_{2,\alpha}$, and $x\in M_{\mathcal{A}}$,
then
$$
(ST)_x=S_xT_x, \quad (TS)_x=T_xS_x, \quad (T^*)_x=T_x^*.
$$
This follows from Propositions \ref{WOTCon} and \ref{SOTCon},
just taking into account that in \eqref{bofz}, $b_z=1$ for $p=2$.
Observe also that if $S\in\mathcal{L}(A^2_\alpha,A^2_{\alpha})$, $z\in \B_n$, and $x\in M_{\mathcal{A}}$, then
$$
\norm{S_x}_{\mathcal{L}(A^2_\alpha,A^2_\alpha)}\leq\norm{S_z}_{\mathcal{L}(A^2_\alpha,A^2_\alpha)}=\norm{S}_{\mathcal{L}(A^2_\alpha,A^2_\alpha)}.
$$

Let $\mathcal{K}$ denote the ideal of compact operators on $A^2_\alpha$.  Recall that the Calkin algebra
is given by $\mathcal{L}(A^2_\alpha, A^2_\alpha)/\mathcal{K}$.
The spectrum of $S$ will be denoted by $\sigma(S)$, and the spectral radius by
$$
r(S)=\sup\{\abs{\lambda}:\lambda\in\sigma(S)\}.
$$
We also define the essential spectrum, $\sigma_e(S)$, as the spectrum of $S+\mathcal{K}$ in the Calkin algebra, and the essential spectral radius as
$$
r_e(S)=\sup\{\abs{\lambda}:\lambda\in\sigma_e(S)\}.
$$
The following result is the improvement that is available in the Hilbert space case.

\begin{thm}
For $S\in\mathcal{T}_{2,\alpha}$ we have
\begin{equation}
\label{EssentialNorm}
\norm{S}_e=\sup_{x\in M_{\mathcal{A}}\setminus\B_n}\norm{S_x}_{\mathcal{L}(A^2_\alpha, A^2_\alpha)}
\end{equation}
and
\begin{equation}
\label{EssentialRad}
\sup_{x\in M_{\mathcal{A}}\setminus\B_n} r(S_x)\leq\lim_{k\to\infty}
\left(\sup_{x\in M_{\mathcal{A}}\setminus\B_n}\norm{S_x^k}^{\frac{1}{k}}_{\mathcal{L}(A^2_\alpha, A^2_\alpha)}\right)=r_e(S),
\end{equation}
with equality when $S$ is essentially normal.
\end{thm}

\begin{proof}
Since $(S^k)_x=(S_x)^k$, then by Theorem \ref{EssentialviaSx} gives
$$
\sup_{x\in M_{\mathcal{A}}\setminus\B_n}\norm{(S_x)^k}^{\frac{1}{k}}_{\mathcal{L}(A^2_\alpha, A^2_\alpha)}\lesssim \norm{S^k}_e^{\frac{1}{k}}\lesssim \sup_{x\in M_{\mathcal{A}}\setminus\B_n}\norm{(S_x)^k}^{\frac{1}{k}}_{\mathcal{L}(A^2_\alpha, A^2_\alpha)}.
$$
Taking the limit as $k\to\infty$ yields
$$
\lim_{k\to\infty}\left(\sup_{x\in M_\mathcal{A}\setminus\B_n}\norm{S_x^k}^{\frac{1}{k}}_{\mathcal{L}(A^2_\alpha, A^2_\alpha)}\right)=r_e(S).
$$
For the inequality one notes that $r(T)\leq \norm{T^k}^{\frac{1}{k}}$ for a generic operator, and consequently
$$
\sup_{x\in M_{\mathcal{A}}\setminus\B_n} r(S_x)\leq \sup_{x\in M_{\mathcal{A}}\setminus\B_n}\norm{(S_x)^k}^{\frac{1}{k}}_{\mathcal{L}(A^2_\alpha, A^2_\alpha)}.
$$
Combining these observations we obtain \eqref{EssentialRad}.  Suppose now that $S$ is essentially normal.
This means that $S^*S-SS^*$ is compact, and therefore
$$
S^*_xS_x-S_xS^*_x=(S^*S-SS^*)_x=0.
$$
Thus, $S_x$ is a normal operator for each $x\in M_{\mathcal{A}}\setminus\B_n$, and
$$
\norm{S_x^k}^{\frac{1}{k}}_{\mathcal{L}(A^2_\alpha, A^2_\alpha)}=r(S_x).
$$
This gives the equality in \eqref{EssentialRad}, since
$$
\sup_{x\in M_\mathcal{A}\setminus\B_n}r(S_x)=\lim_{k\to\infty}\sup_{x\in M_\mathcal{A}\setminus\B_n}\norm{S_x^k}^{\frac{1}{k}}_{\mathcal{L}(A^2_\alpha, A^2_\alpha)}=r_e(S).
$$
Now apply the equality in \eqref{EssentialRad} to the operator $S^*S$ and note that
\begin{eqnarray*}
\norm{S}^2_e=\norm{S^*S}_e=r_e(S^*S) & = & \sup_{x\in M_\mathcal{A}\setminus\B_n}r((S^*S)_x)\\
 & = & \sup_{x\in M_\mathcal{A}\setminus\B_n}\norm{S_x^*S_x} _{\mathcal{L}(A^2_\alpha, A^2_\alpha)}\\
 & = & \sup_{x\in M_\mathcal{A}\setminus\B_n}\norm{S_x}^2 _{\mathcal{L}(A^2_\alpha, A^2_\alpha)}.
\end{eqnarray*}
\end{proof}
The following Corollary can be proved in a similar manner as in \cite{Sua}.
\begin{cor}
Let $S\in \mathcal{T}_{2,\alpha}$ and $\eta, \delta\in\R$ be such that $\eta I\leq S_x\leq \delta I$
for all $x\in M_{\mathcal{A}}\setminus\B_n$.  Then given $\epsilon>0$, there is a compact self-adjoint operator $K$ such that
$$
(\eta-\epsilon)I\leq S+K\leq(\delta+\epsilon)I.
$$
\end{cor}
Using the tools from above, and repeating the proof in \cite{Sua} we have the following.
\begin{thm}
Let $S\in\mathcal{T}_{2,\alpha}$.  The following are equivalent:
\begin{itemize}
\item[(1)] $\lambda\notin \sigma_e(S)$;
\item[(2)] $$\lambda\notin \bigcup_{x\in M_\mathcal{A}\setminus\B_n}\sigma(S_x)\quad\textnormal{ and }\quad\sup_{x\in M_\mathcal{A}\setminus\B_n}\norm{\left(S_x-\lambda I\right)^{-1}}_{\mathcal{L}(A^2_\alpha,A^2_\alpha)}<\infty;$$
\item[(3)] there is a number $t>0$ depending only on $\lambda$ such that
$$
\norm{(S_x-\lambda I)f}_{A^2_\alpha}\geq t\norm{f}_{A^2_\alpha}\quad\textnormal{ and }\quad\norm{(S_x^*-\overline{\lambda} I)f}_{A^2_\alpha}\geq t\norm{f}_{A^2_\alpha}
$$
for all $f\in A^2_\alpha$ and $x\in M_{\mathcal{A}}\setminus\B_n$.
\end{itemize}
\end{thm}
The above theorem then yields the following Corollary.
\begin{cor}
If $S\in\mathcal{T}_{2,\alpha}$ then
$$
\overline{\bigcup_{x\in M_\mathcal{A}\setminus\B_n}\sigma(S_x)}\subset\sigma_e(S),
$$
with equality if $S$ is essentially normal.
\end{cor}




\begin{bibdiv}
\begin{biblist}

\bib{AZ}{article}{
   author={Axler, Sheldon},
   author={Zheng, Dechao},
   title={Compact operators via the Berezin transform},
   journal={Indiana Univ. Math. J.},
   volume={47},
   date={1998},
   number={2},
   pages={387--400}
}

\bib{AZ2}{article}{
   author={Axler, Sheldon},
   author={Zheng, Dechao},
   title={The Berezin transform on the Toeplitz algebra},
   journal={Studia Math.},
   volume={127},
   date={1998},
   number={2},
   pages={113--136}
}

\bib{BI}{article}{
   author={Bauer, Wolfram},
   author={Isralowitz, Joshua},
   title={Compactness characterization of operators in the Toeplitz algebra of the Fock space $F_\alpha^p$},
   journal={J. Funct. Anal.},
   volume={263},
   date={2012},
   number={5},
   pages={1323--1355},
   eprint={http://arxiv.org/abs/1109.0305v2}
}

\bib{B}{article}{
   author={Berndtsson, Bo},
   title={Interpolating sequences for $H^\infty$ in the ball},
   journal={Nederl. Akad. Wetensch. Indag. Math.},
   volume={47},
   date={1985},
   number={1},
   pages={1--10}
}

\bib{CR}{article}{
   author={Coifman, R. R.},
   author={Rochberg, R.},
   title={Representation theorems for holomorphic and harmonic functions in
   $L^{p}$},
   conference={
      title={Representation theorems for Hardy spaces},
   },
   book={
      series={Ast\'erisque},
      volume={77},
      publisher={Soc. Math. France},
      place={Paris},
   },
   date={1980},
   pages={11--66}
}

\bib{E92}{article}{
   author={Engli{\v{s}}, Miroslav},
   title={Density of algebras generated by Toeplitz operator on Bergman spaces},
   journal={Ark. Mat.},
   volume={30},
   date={1992},
   pages={227--243}
}

\bib{E}{article}{
   author={Engli{\v{s}}, Miroslav},
   title={Compact Toeplitz operators via the Berezin transform on bounded
   symmetric domains},
   journal={Integral Equations Operator Theory},
   volume={33},
   date={1999},
   number={4},
   pages={426--455}
}

\bib{I}{article}{
   author={Issa, Hassan},
   title={Compact Toeplitz operators for weighted Bergman spaces on bounded
   symmetric domains},
   journal={Integral Equations Operator Theory},
   volume={70},
   date={2011},
   number={4},
   pages={569--582}
}

\bib{LH}{article}{
   author={Li, Song Xiao},
   author={Hu, Jun Yun},
   title={Compact operators on Bergman spaces of the unit ball},
   language={Chinese, with English and Chinese summaries},
   journal={Acta Math. Sinica (Chin. Ser.)},
   volume={47},
   date={2004},
   number={5},
   pages={837--844}
}

\bib{L}{article}{
   author={Luecking, Daniel H.},
   title={Representation and duality in weighted spaces of analytic
   functions},
   journal={Indiana Univ. Math. J.},
   volume={34},
   date={1985},
   number={2},
   pages={319--336}
}

\bib{NZZ}{article}{
   author={Nam, Kyesook},
   author={Zheng, Dechao},
   author={Zhong, Changyong},
   title={$m$-Berezin transform and compact operators},
   journal={Rev. Mat. Iberoam.},
   volume={22},
   date={2006},
   number={3},
   pages={867--892}
}

\bib{R}{article}{
   author={Raimondo, Roberto},
   title={Toeplitz operators on the Bergman space of the unit ball},
   journal={Bull. Austral. Math. Soc.},
   volume={62},
   date={2000},
   number={2},
   pages={273--285}
}

\bib{St}{article}{
   author={Stroethoff, Karel},
   title={Compact Toeplitz operators on Bergman spaces},
   journal={Math. Proc. Cambridge Philos. Soc.},
   volume={124},
   date={1998},
   number={1},
   pages={151--160}
}

\bib{SZ}{article}{
   author={Stroethoff, Karel},
   author={Zheng, Dechao},
   title={Toeplitz and Hankel operators on Bergman spaces},
   journal={Trans. Amer. Math. Soc.},
   volume={329},
   date={1992},
   number={2},
   pages={773--794}
}

\bib{Sua}{article}{
   author={Su{\'a}rez, Daniel},
   title={The essential norm of operators in the Toeplitz algebra on $A^p(\mathbb{B}_n)$},
   journal={Indiana Univ. Math. J.},
   volume={56},
   date={2007},
   number={5},
   pages={2185--2232}
}

\bib{Sua2}{article}{
   author={Su{\'a}rez, Daniel},
   title={Approximation and the $n$-Berezin transform of operators on the
   Bergman space},
   journal={J. Reine Angew. Math.},
   volume={581},
   date={2005},
   pages={175--192}
}

\bib{YS}{article}{
   author={Yu, Tao},
   author={Sun, Shan Li},
   title={Compact Toeplitz operators on the weighted Bergman spaces},
   language={Chinese, with English and Chinese summaries},
   journal={Acta Math. Sinica (Chin. Ser.)},
   volume={44},
   date={2001},
   number={2},
   pages={233--240}
}

\bib{Zhu}{book}{
   author={Zhu, Kehe},
   title={Spaces of holomorphic functions in the unit ball},
   series={Graduate Texts in Mathematics},
   volume={226},
   publisher={Springer-Verlag},
   place={New York},
   date={2005},
   pages={x+271}
}

\end{biblist}
\end{bibdiv}
\end{document}